\documentclass[11pt]{article}%
\usepackage{algpseudocode}
\usepackage{float} 
\usepackage{fullpage}
\usepackage{amsfonts}
\usepackage{amsmath}
\usepackage{amssymb}
\usepackage{amsbsy}
\usepackage{amsthm}
\usepackage{mathtools}
\usepackage{epsfig}
\usepackage{graphicx}
\usepackage{colordvi}
\usepackage{graphics}
\usepackage{color}
\usepackage{verbatim} 
\numberwithin{equation}{section}
\newtheorem{theorem}{Theorem}[section]
\newtheorem{thm}{Theorem}[section]

\newtheorem{lemma}{Lemma}[section]

\newtheorem{prop}[thm]{Proposition}
\newtheorem{alg}[thm]{Algorithm}
\newtheorem{remark}[thm]{Remark}

\newcommand{\norm}[1]{\left\Vert#1\right\Vert}

\newcommand{\pa}{\partial}

\newcommand{\nr}[1]{\ensuremath{\left\|{#1} \right\|}}
\newcommand{\grad}{\nabla}
\newcommand{\inv}[1]{{#1}^{-1}}

\newcommand{\f}{\frac}

\newcommand{\cF}{\mathcal F}

\newcommand{\invR}{R^{-1}}
\newcommand{\eps}{\epsilon}
\DeclareMathOperator{\argmin}{argmin}
\DeclareMathOperator{\cond}{cond}
\DeclareMathOperator{\spa}{span}
\DeclareMathOperator{\divi}{div}
\DeclareMathOperator{\dd}{\,d}

\begin{document}

\title{Filtering for Anderson acceleration}

\author{
Sara Pollock
\thanks{Department of Mathematics,  University of Florida, Gainesville, FL 32611 (s.pollock@ufl.edu)}
\and
Leo G. Rebholz
\thanks{School of Mathematical and Statistical Sciences, Clemson University, Clemson SC 29634 (rebholz@clemson.edu)}
}
\date{}
\maketitle
\begin{abstract}
This work introduces, analyzes and demonstrates an efficient and theoretically sound 
filtering strategy to ensure the condition of the least-squares problem solved at
each iteration of Anderson acceleration. The filtering strategy consists of two steps:
the first controls the length disparity between columns of the least-squares matrix, 
and the second enforces a lower bound on the angles between subspaces spanned by the
columns of that matrix.  The combined strategy is shown to control the condition number
of the least-squares matrix at each iteration. The method is shown to be effective on a 
range of problems based on discretizations of partial differential equations. It is shown
particularly effective for problems where the initial iterate may lie far from the
solution, and which progress through distinct preasymptotic and asymptotic phases.
\end{abstract}

\section{Introduction}\label{sec:intro}
Nonlinear problems are ubiquitous throughout physical modeling, mathematics and
data sciences.  Many basic iterative methods for solving such problems can be written
in the form of a fixed-point iteration: $x_{k+1}= g(x_k)$, where a sequence of iterates
$x_k$ form approximations to the fixed-point solution $x$.
Anderson acceleration (AA) has become an increasingly popular method
for decreasing the number of fixed-point iterations to convergence, and in many cases
enabling convergence where the original fixed-point iteration fails.  
It involves a correction  to each fixed-point update formed from a linear 
combination of a history of iterates and update steps, where the linear combination
is chosen so that the corrected update step has least length.
Its popularity is due in large part to how effective it is for problems over a wide 
range of fields, including, to name a few, 
quantum chemistry, physics, multiphysics and flow phenomena 
\cite{AJW17,EPRX20,FaSa09,LWWY12,PRX19,TKSHCP15}, and recently 
data science and optimization \cite{PDZGQL18,SSWHWH19,WHS21}.

AA was introduced in 1965 by D. Anderson in \cite{Anderson65} in the context of
integral equations, analyzed as a
generalized multi-secant or quasi-Newton algorithm in \cite{E96,FaSa09}, and 
discussed within a Krylov space framework and popularized by its effective and efficient 
application across a variety of problems in \cite{WaNi11}.
Local convergence of the iteration was first shown in \cite{ToKe15}, and the acceleration
property of the method was first theoretically developed in \cite{EPRX20,PR21,PRX19}.
The method is well known to be sensitive to implementation and parameter choices,
including those for the relaxation factor at each iteration, 
the (maximal) algorithmic depth, 
as well as whether and how often the method should be restarted.
Here we introduce a stabilization through selectively eliminating steps from the
history in order to improve convergence and control the condition of the inner 
optimization.

Herein we seek a fixed-point $x$ of $g(x)$ by a fixed-point iteration 
$x_{k+1} = g(x_k)$. Define the 
difference between iterates $e_k$ and the residual $w_k$ by
\begin{align*}
e_{k}  \coloneqq x_k - x_{k-1}, \quad
w_k  \coloneqq g(x_{k-1}) - x_{k-1}.
\end{align*}
Then AA can be written as follows, where parameter $m$ is the maximum allowable
algorithmic depth, $m_k$ is the algorithmic depth at iteration $k$, and $\beta_k$
is the 
relaxation (also called damping or mixing) parameter used on iteration $k$. 

\begin{alg}[Anderson acceleration]\label{alg:aa}
The algorithm starts with a fixed point update step ($k=0$),
then the acceleration starts at $k=1$.
\begin{algorithmic}
\State Choose initial iterate $x_0$, and algorithmic depth parameter $m$
\State Compute $w_1$, 
set $\beta_0$, and update $x_1 = x_0 + \beta_0 w_1$
\Comment{$k=0$}
\end{algorithmic}
\begin{algorithmic}[1]
\For{$k = 1, \ldots$}
\Comment{$k > 0$}
\State Compute $w_{k+1}$
\State Set $m_k = \min\{k,m\}$
\State 
Set $F_k= \begin{pmatrix}(w_{k+1}-w_k) & \hdots & 
                         (w_{k+1-(m_k-1)} - w_{k-(m_k-1)})\end{pmatrix}$\\
\hspace{.4in} and  $E_k= \begin{pmatrix}(e_{k}) & \ldots &
                         (e_{k-(m_k-1)})\end{pmatrix}$
\State
Find $\gamma_k = \argmin  \|F_k \gamma - w_{k+1}\|$
\State Set relaxation parameter $\beta_k$
\State
Update $x_{k+1} = x_k + \beta_k w_{k+1} - \left(E_k + \beta_k F_k \right)\gamma_{k+1}$
\EndFor
\end{algorithmic}
\end{alg}
Define the optimization gain $\theta_{k+1}$ by
\begin{align}\label{eqn:gaindef}
\theta_{k+1} \coloneqq \f{\norm{F_k \gamma_k - w_{k+1}}}{\norm{w_{k+1}}}.
\end{align}
As shown in \cite{EPRX20,PR21,PRX19}, at each iteration $k$, the first order term
in the residual is improved (in comparison to a fixed-point iteration with the same
damping factor $\beta_k$) by a factor of $\theta_k$, but at the cost of additional
higher-order terms.

It is typical but not required to interpret the minimization problem in line 5
of algorithm \ref{alg:aa} 
as least-squares in either the $l_2$ or a weighted $l_2$ norm 
\cite{EPRX20,FaSa09,PR21,WaNi11,YTA22}.  
However, other norms such as $l_1$ or $l_\infty$ can be used \cite{ToKe15}. Here
we will restrict our attention to interpeting the minimization as a (weighted)
least-squares problem, which can be efficiently solved using for instance a QR 
factorization as in \cite{WaNi11}. This finite dimensional Hilbert space setting is 
used to further understand the theoretical convergence properties of AA, and to 
improve performance through dynamic parameter selection in \cite{PR21}. 

The main issue addressed herein is controlling the condition of the matrix $F_k$
used in the least-squares problem in line 5, 
in order to improve the numerical stability of the algorithm.
To this end we introduce a column filtering
strategy to efficiently control the condition of $F_k$.  
In \cite{FaSa09}, both truncated singular value decomposition (TSVD) and Householder
QR with column pivoting (rank revealing QR) are discussed as methods to address poorly 
conditioned least-squares problems that may arise in AA.  
Both of these standard methods for rank deficient and ill-conditioned least-squares 
problems may ultimately interfere with the convergence of AA as the columns of the 
iteration-$k$ least-squares matrix $F_k$
have a natural ordering based on the algorithmic age of the information they represent.
The rank revealing QR approach can change the order of the columns, keeping older 
information which may pollute the solution even while improving conditioning; and, the 
TSVD approach may reweight the columns without discarding older information to 
similar effect.
Another approach to this problem is to add a Tikhonov-type regularization term to the 
least-squares problem. This approach is used for instance in \cite{SdAB16} in the 
context of unconstrained optimization and in \cite{SSWHWH19}, 
in the context of deep reinforcement learning.
While this regularization can be used to control the size of the optimization 
coefficients, it can potentially lead to the selection of less
effective coefficients.

We will use use the more stable TSVD approach 
as a standard for comparison in our numerical tests. 
We will demonstrate in section \ref{sec:numerics} 
that our presently proposed filtering algorithm which features a low overall additional 
computational complexity compares well in terms of the number
of iterations to convergence while controlling the condition numbers; often converging
in substantially fewer iterations or converging where the TSVD approach fails.
The filtering approach further appears more robust with respect to parameter 
selection. 

The remainder of the manuscript is organized as follows.
In subsection \ref{sec:bgt} we review the background theory that supports each of 
the two parts of the filtering method: angle filtering and length filtering.
In section \ref{sec:filter} we introduce the filtering algorithms, which efficiently
control the condition of $F_k$ at each iteration $k$.
In section \ref{sec:ftheory} we analyze the strategy and establish that the combination
of the two filtering routines, one controlling small angles between columns and the
other controlling relative magnitudes of columns, is sufficient to control the 
condition.
In section \ref{sec:numerics} the filtering strategy is tested against the 
standard TSVD on a number of problems with varying complexity.

In the remainder, let $\nr{A}_F$ denote the Frobenius norm of $A$,
the matrix with columns $a_1, \ldots, a_m$.
Vector norms without a subscript, {\em e.g.,} $\nr{a_i}$, 
will denote the vector 2-norm.

\subsection{Background theory}\label{sec:bgt}
The proposed filtering method can be thought of as combining two approaches, each
of which partially controls the condition of $F_k$, the matrix that 
arises in the least-squares problem at each iteration of AA. 
The first is the angle filtering approach of 
\cite{PR21} which controls the angle between each column of $F_k$ and the subspace
spanned by the columns to its left (which contain more recent information). 
Specifically, define $\cF_j$ as the subspace spanned by the first $j$ columns of a
given matrix $F_k$ and define the direction sine $\sigma_i$ by 
\begin{align}\label{eqn:defsigma}
\sigma_i = \sin(\cF_{i-1}, f_i), ~i = 2, \ldots, m_k, ~\text{ and } 
\sigma_{k,min} = \min_{i = 2, \ldots, m_k} \sigma_i,
\end{align}
where $f_i$ is column $i$ of matrix $F_k$.  Given a QR factorization $F_k = QR$, 
we have $\sigma_i = |r_{ii}|/\norm{f_i}$.  This quantity can then be monitored and
used to filter out columns of $F_k$ at each iteration $k$.
The second approach is that of \cite{AJW17,WaNi11} which 
sequentially drops the oldest columns of $F_k$, updating the QR decomposition and 
checking the condition number until the condition is within a given bound.  
Neither of these strategies alone is sufficient to efficiently control the condition
of $F_k$ as the condition can become high due to near linear dependence of the columns
either through disparity in lengths or alignment in angle.
This work combines elements from both strategies to give a guaranteed bound on the 
condition number after at most a single update of the QR decomposition, while
maintaining the angle condition which is necessary for the convergence analysis
of \cite{PR21} for smooth problems. As we show numerically in section 
\ref{sec:numerics}, filtering is beneficial to convergence in both the preasymptotic
and asymptotic regimes.

As in \cite{AJW17,WaNi11}, one part of this filtering method  drops 
older information from the system, which reduces the build-up of 
higher-order (nonlinear) residual terms that can be seen in the analysis of \cite{PR21}.
Assuming the fixed-point operator is sufficiently smooth, with Lipschitz constant
$\kappa_g$ and Lipschitz constant of its derivative $\hat \kappa_g$,
Theorem 5.1 of \cite{PR21} proves that the step $k$ residual $w_{k+1}$ satisfies
\begin{align}\label{eqn:tgenm}
\nr{w_{k+1}} & \le \nr{w_k} \Bigg\{
 \theta_k ((1-\beta_{k}) + \kappa_g \beta_{k})
+ C_k \hat \kappa_g  \sqrt{1-\theta_k^2} 
   \sum_{n = k-{m_{k-1}}}^{k} 
\nr{w_n} 
  \Bigg\},
\end{align}
where at each iteration $k$, the quantity $C_k$ depends on the previous $m_k$ values
of $\theta_j$ and $\beta_j$ (introduced in algorithm \ref{alg:aa} and 
\eqref{eqn:gaindef}), and increases as each of the previous $m_k$ values of 
$\sigma_{j,min}$ from \eqref{eqn:defsigma} decreases. 
From this estimate, we observe that in a converging iteration the largest of the 
higher order terms will be from older information, i.e. 
$\| w_k \| \| w_{k-m_{k-1}} \|$, and can be
(possibly significantly) larger than $\| w_k \|^2$.  
Hence strategies to keep $C_k$ small should also weigh the effect that older 
information may have on the residual.

The length filtering algorithm introduced in subsection \ref{subsec:lfiltc}
is based on columnwise bounds for the condition number of a matrix $F$, assuming a lower
bound on its minimal direction sine between columns as defined by \eqref{eqn:defsigma}.
Computing the Frobenius norm of $F$ columnwise
is both straightforward and sharp. Our bounds for $\nr{F^{-1}}_F = \nr{R^{-1}}_F$ for
$F = QR$ are based on the two following results for general rectangular matrices.

The first is a componentwise bound on the entries of $\invR$ from \cite{PR21}.
\begin{lemma}[\cite{PR21}, Lemma 5.1]\label{lem:Rinv}
Let $F = QR$ be an economy QR decomposition of $n \times m$ matrix $F$, with
$n \ge m\ge 2$.
Let $\cF_p = \spa\{f_1, \ldots, f_p\}$, the subspace spanned by the 
first $p$ columns of $F$.
Suppose  there is a constant $0 < c_s \le 1$ such that $\sin(f_i,\cF_{i-1}) \ge c_s$,
by which there is another constant $c_t = \sqrt{1 - c_s^2}$ for which
$\cos(f_i,f_k)\le c_t$, for $k = 1, \ldots, i$.
Denote $R^{-1} = (s_{ij})$.
Then it holds that
\begin{align}\label{eqn:invR1}
s_{11} &= \f {1}{\nr{f_1}},
\quad &&
|s_{1j}| \le \f{c_t(c_t+c_s)^{j-2}}{\nr{f_1}c_s^{j-1}}, ~2 \le j \le m,
~\text{and}
\\ \label{eqn:invRi}
s_{ii} &\le \f{1}{\nr{f_i}c_s}, ~2\le i \le m,
&&
|s_{ij}| \le \f{c_t(c_t+c_s)^{j-i-1}}{\nr{f_i}c_s^{j-i+1}},
~i+1 \le j \le m.
\end{align}
\end{lemma}

Next, summing by column over the square of the bounds from lemma \ref{lem:Rinv}
gives the following bounds on the squared $l_2$ norm of each column of $\inv R$.
\begin{prop}\label{prop:colbd}
Suppose the hypotheses of lemma \ref{lem:Rinv}.
Let $s_i$ denote column $i$ of $\inv R$.  The the following bounds hold:
\begin{align}\label{eqn:colbd1}
\nr{s_1}^2 & = \f{1}{\nr{f_1}^2} \eqqcolon b_1,
\\ \label{eqn:colbd2}
\nr{s_2}^2 &\le \f{1}{c_s^2}\left\{\f{c_t^2}{\nr{f_1}^2} + \f{1}{\nr{f_2}^2} \right\}
\eqqcolon b_2,
\\ \label{eqn:colbdj}
\nr{s_j}^2 &\le \f{1}{c_s^2}\left( \f{c_t^2(c_t + c_s)^{2(j-2)}}{\nr{f_1}^2 c_s^{2(j-2)}} 
+ \sum_{i = 2}^{j-1}\f{c_t^2(c_t + c_s)^{2(j-i-1)}}{\nr{f_i}^2c_s^{2(j-i)}}
+ \f{1}{\nr{f_j}^2} \right\} \eqqcolon b_j , \ \ 3\le j \le  m.
\end{align}
\end{prop}

In the length filtering algorithm, the sums of the squared norms of the
first $l$ columns of $f$, multiplied by the sum over $b_j, ~j = 1, \ldots, l$, as defined
in \eqref{eqn:colbd1}-\eqref{eqn:colbdj}, are used to bound the square of the 
Frobenius condition number of $F$, truncated to its first $l$ columns.

In summary, the (angle) filtering algorithm introduced in \cite{PR21} removes columns 
of $F_k$ which are close in angle to the subspace spanned by more recent columns.  
This controls the 
constant $C_k$ that multiplies
higher order residual terms of \eqref{eqn:tgenm} as shown in
\cite[Theorem 5.1]{PR21}. 
However, alone this method is not sufficient to control the condition 
of the least-squares matrix $F_k$, since poor condition can also 
arise
from large magnitude differences between the lengths
of the columns of $F_k$. 
The full condition filtering algorithm presented here 
consists of the length filtering algorithm \ref{alg:lfilt} followed by the
angle filtering algorithm \ref{alg:afilt}. 

\section{Filtered AA}\label{sec:filter}
The filtered Anderson acceleration (FAA) algorithm presented here consists of two 
filtering steps, the first filtering for disparity in column lengths, 
and the second filtering for small angles between each column and the subspace spanned 
by the columns to its left.  The filtered acceleration algorithm is presented in
algorithm \ref{alg:faa}. The individual filtering routines are described in algorithms
\ref{alg:afilt} and the new algorithm \ref{alg:lfilt}.  
The algorithm removes columns from $F_k$ at each iteration $k > 1$
to ensure $\cond_F( F_k) \le \bar \kappa$, where $\bar \kappa$ is a user chosen upper bound on the allowable condition number. 

The FAA algorithm \ref{alg:faa} requires two additional parameters in comparison to the 
standard AA algorithm \ref{alg:aa}. The first is $\bar \kappa$, a user-determined
upper-bound on the condition of each matrix $F_k$ used in the least-squares problem.
The second is $c_s$, a lower bound on the allowable direction sines between each 
column of each matrix $F_k$ and the subspace spanned by columns with more recent
information, as defined in \eqref{eqn:defsigma}. A discussion on parameter
selection is included in subsection \ref{subsec:ang}.

\begin{alg}[Filtered Anderson acceleration)]\label{alg:faa}
The algorithm starts with a fixed point update step ($k=0$), an AA(1) step ($k=1$), 
then the filtering starts at $k=2$.
\begin{algorithmic}
\State Choose initial iterate $x_0$, and parameters $c_s, ~\bar \kappa$ and $m$
\State Compute $w_1$, set $\beta_0$ and update $x_1 = x_0 + \beta_0 w_1$ 
\Comment{$k=0$}
\State Compute $w_2$ 
\Comment{$k=1$}
\State 
Set $F_1= \begin{pmatrix} w_{2}-w_1 \end{pmatrix}$
and 
$E_1= \begin{pmatrix} x_{1}-x_{0}\end{pmatrix}$
\State Compute $\gamma_2$ by $F_1^\ast F_1 \gamma_2 = F_1^\ast w_2$
\State
Set $m_1 = 1$, set $\beta_1$ and 
update $x_{2} = x_1 + \beta_1 w_{2} - \left(E_1 + \beta_1 F_1 \right)\gamma_{2}$
\end{algorithmic}
\begin{algorithmic}[1]
\For{$k = 2, \ldots$}
\Comment{$k > 1$}
\State Compute $w_{k+1}$
\State Update $m_k = \min\{m_{k-1}+1,m\}$, and drop the last columns of 
$E_{k-1}$,$F_{k-1}$ if $m_{k-1}=m$ 
\State 
Set $F_k= \begin{pmatrix}(w_{k+1}-w_k) & (F_{k-1})\end{pmatrix}$
and  $E_k= \begin{pmatrix}(x_{k}-x_{k-1}) & (E_{k-1})\end{pmatrix}$
\State $[E_k,F_k,m_k] =$ Length Filter ($E_k,F_k,m_k,c_s, \bar \kappa$)
\State $[E_k,F_k,Q_k,R_k,m_k] =$ Angle Filter ($E_k,F_k,m_k,c_s$)
\State Compute $\gamma_{k+1}$
by $R_k \gamma_{k+1}= Q_k^\ast w_{k+1}$
\State
Set $\beta_k$
\State
Update $x_{k+1} = x_k + \beta_k w_{k+1} - \left(E_k + \beta_k F_k \right)\gamma_{k+1}$
\EndFor
\end{algorithmic}
\end{alg}

\begin{remark}[TSVD]\label{rem:tsvd}
We compare our numerical results with algorithm \ref{alg:faa} against the standard
method of controlling the condition number by TSVD, as suggested in \cite{FaSa09}. 
Following \cite{chan82} we implement the TSVD by computing the economy QR 
factorization $F_k = Q R$, followed by the SVD of $R = U_R \Sigma_R V_R^\ast$,
determine the largest $s$ such that the ratio of singular values 
$\sigma_1 / \sigma_s < \bar \kappa$, then restrict $U_R$ and $V_R$ to their first 
$s$ columns and $\sigma_R$ to its first $s$ rows and columns, referred to respectively
as $U_s, V_s$ and $\Sigma_s$.
The solution of the least-squares problem 
$\gamma_{k+1} = \argmin \nr{F_k \gamma - w_{k+1}}$ is then given by 
$\gamma_{k+1} = V_s \Sigma_s^{-1} U_s^\ast Q^\ast w_{k+1}$.
This procedure replaces lines 5-7 of algorithm \ref{alg:faa}.
\end{remark}

\begin{remark}[On the order of length and angle filtering]\label{rem:order}
Algorithm \ref{alg:faa} does length filtering before angle filtering to potentially
reduce the size of the QR factorization. However the order of the filtering steps 
could be reversed provided an additional step to truncate the factors of the QR 
decomposition is performed after the length filtering. 
Each order has its own advantage, as discussed in remark \ref{rem:order2}, 
which specifies how Theorem \ref{thm:condbd} still holds in the 
case of angle filtering performed first.
\end{remark}

\subsection{Angle filtering}\label{subsec:ang}
The angle filtering algorithm was proposed by the authors in \cite{PR21} in order to
determine a one-step residual bound for smooth problems, 
free from an assumed $l_\infty$ bound on the optimization coefficients.  

\begin{alg}[$E,F,Q,R,m$ = Angle filtering($E,F,m,c_s$)] \label{alg:afilt}
Filter out columns of $F$ by direction sine.
\begin{algorithmic}[1]
\State Compute the economy QR decomposition $F = Q R$
\State Compute $\sigma_i = |r_{ii}|/\|f_i\|, ~ i = 2,\ldots, m$, 
where $f_i$ is column $i$ of $F$, and
$r_{ii}$ is the corresponding diagonal entry of $R$
\State Remove any columns $i$ of $E$ and $F$ for which $\sigma_i < c_s$
\State Update $m$ with the new number of columns of $F$
\If{any columns were removed}
\State Recompute $F = QR$
\EndIf
\end{algorithmic}
\end{alg}

Algorithm \ref{alg:faa} has two additional parameters to choose, in comparison to 
standard AA. One is $\bar \kappa$, the maximum allowable condition number.
The second is $c_s$, the minimum allowable direction sine between columns of $F_k$.
The condition estimate in the length-filtering algorithm is less sharp for smaller
values of $c_s$.  However, if $c_s$ is chosen too large, the angle filtering algorithm
\ref{alg:afilt} can reduce the algorithm to standard AA with $m=1$ (see
subsection \ref{subsec:quasi}, below).  
We generally found a range for $c_s$ between $0.1$ and $1/\sqrt 2$ to work well. 
In problems where the solution progresses through distinct preasymptotic and asymptotic 
regimes, as shown below in subsection \ref{subsec:plap}, it can also be advantageous to 
start with a higher values of $c_s$ and dynamically decrease it to a lower value.
Heuristically this makes sense because a
more stringent angle filtering condition can be advantageous in the preasymptotic regime,
as it prevents the build-up of higher-order residual terms. However in the asymptotic 
regime, a smaller value of $c_s$ may be preferred because the more recent information
is kept, and older columns in $F_k$ with orders of magnitude difference in length 
from the first are discarded. Our numerical examples show that the filtering 
algorithm \ref{alg:faa} also decreases the sensitivity of AA to the choice of maximum 
algorithmic depth $m$, as $m$ can be chosen quite large, and extra columns will simply
be (length) filtered out.

\subsection{Length filtering algorithm}\label{subsec:lfiltc}

Here we present a novel length filtering strategy. It is 
motivated by the bounds on columns of $R^{-1}$ from proposition \ref{prop:colbd}.
In this algorithm, the Frobenius norm of the
first $k$ columns of $F$ multiplied by the sum over the bounds of the squared norms of
each column of $R^{-1}$, denoted $b_j, ~j = 1, \ldots, k$, as defined
in \eqref{eqn:colbd1}-\eqref{eqn:colbdj}, are used to bound the square of the 
Frobenius condition number of $F$, denoted in the algorithm as $C_F$.  From $m_k$ down
to $1$, the constant $C_F$ determined by the first $k$ columns of $F$ is updated, 
until it is less than a desired (user-chosen) number. On terminating this loop with
some number of $\bar k \ge 1$ columns remaining, the columns $j > \bar k$ are removed from
the matrices $E$ and $F$.

The combination of length and angle filtering is further
justified in the next section by the main theorem \ref{thm:condbd}, which
shows the combination of length filtering algorithm \ref{alg:lfilt} 
and angle filtering algorithm \ref{alg:afilt} controls the Frobenius condition number 
of the least-squares matrix $F_k$ at each iteration $k$ of FAA algorithm \ref{alg:faa}.

\begin{alg}[$E,F, m$ = Length filtering($E,F,m,c_s$)]\label{alg:lfilt} 
Filter out columns of $F$ by length.
\begin{algorithmic}[1]
\State Define $c_t = \sqrt{1-c_s^2}$ 
\State Compute the vector norms $\nr{f_j}^2$ and the bounds $b_j$ for $\nr{s_j}^2$, 
given by \eqref{eqn:colbd1}-\eqref{eqn:colbdj}, for $j = 1, \ldots, m$ 
\For {$k = m$ down to $1$}
  \State Compute $C_F =
\left( \sum_{j = 1}^k \nr{f_j}^2 \right)\left( \sum_{j = 1}^k b_j \right)$
  \State Exit loop if $C_F \le \bar \kappa^2$
\EndFor
\State $m \gets k$, $E \gets E(:,1:m)$, $F \gets F(:,1:m)$
\end{algorithmic}
\end{alg}

\section{Theory}\label{sec:ftheory}
Together, as shown below in theorem \ref{thm:condbd}, the two low-cost 
filtering strategies, 
algorithms \ref{alg:afilt} and \ref{alg:lfilt},
bound the Frobenius condition number of each least-squares matrix that arises in AA.
Let $F = QR$ be an economy QR decomposition of $n \times m$ matrix $F$. 
Then $\nr{R}_F = \nr{F}_F$, and $\nr{R}_2 = \nr{F}_2$.
By the equivalence $\nr{F}_2 \le \nr{F}_F \le \sqrt m \nr {F}_2$ 
\cite[Chapter 2]{GoVL96}, controlling the Frobenius condition number of $F$ 
controls the $2$-norm condition number of $F$ as well.

The main theoretical result is that given a minimum direction sine $0 < c_s \le 1$ and 
a maximum condition number $\bar \kappa$, the combination of algorithms 
\ref{alg:afilt} and \ref{alg:lfilt} ensures the matrix
$F_k$ from algorithm \ref{alg:faa} has Frobenius condition number no greater than
$\bar \kappa$ at each iteration $k$.
We are concerned with the case $n \gg m$, meaning the number of degrees of freedom
is some orders of magnitude larger than the algorithmic depth $m = m_k$. 
In this setting,
an additional computational cost of algorithm \ref{alg:lfilt} is negligible.

\begin{theorem}\label{thm:condbd}
Suppose the FAA algorithm \ref{alg:faa} is run with a given 
$0 < c_s < 1$ and $\bar \kappa > 1$. Then at each iteration $k$, the matrix $F_k$
of the least-squares problem solved in step 7 of the FAA algorithm \ref{alg:faa}
by the QR decomposition has Frobenius hence $2$-norm condition number less than 
$\bar \kappa$. Moreover, both filtering steps are guaranteed to keep the 
first column of $F_k$, preserving the use of the most recent information.
\end{theorem}
\begin{proof}
By the standard inequality $\nr{F}_2 \le \nr{F}_F$ \cite[Chapter 2]{GoVL96}, it is 
sufficient to establish the result for the Frobenius condition number.
First we have 
\begin{align}\label{eqn:bound1}
\cond(F)_F^2
 = \nr{F}_F^2 \nr{F^{-1}}_F^2  
= \nr{F}_F^2 \nr{R^{-1}}_F^2 
= \left( \sum_{i = 1}^m \nr{f_i}^2 \right)\nr{R^{-1}}_F^2,
\end{align}
where $m$ is the number of columns of $F$.

Suppose the length filtering algorithm \ref{alg:lfilt} is done first. Then
since the bounds on column $j$ of $R^{-1}$ in Proposition \ref{prop:colbd} given by 
\eqref{eqn:invR1}-\eqref{eqn:invRi} depend only on 
information with indices $p \le j$,
removing columns $p$ of $F$ for which $p > j$ does not change the column 
sum of index $j$.  Hence the matrix $\hat F$ with $\hat m$ columns returned from length 
filtering algorithm \ref{alg:lfilt} and sent to angle filtering algorithm 
\ref{alg:afilt} would have condition number less that $\bar \kappa$, if it satisfied the 
angle conditions 
$\sin(f_i, \cF_{i-1})$, 
$i = 2, \ldots, \hat m$. 
The first index $i = 1$ is not included in the previous statement as the angle condition 
is satisfied on the first column by default. 
Moreover, algorithm \ref{alg:lfilt} will never remove the first column of $F$ 
so long as $\bar \kappa \ge 1$ as $k=1$ yields $C_F =1$.
The angle filtering algorithm \ref{alg:afilt} may remove any of columns 
$2, \ldots \hat m$ of $\hat F$, which would only decrease the bounds in 
proposition \ref{prop:colbd} by removing terms.  Hence the condition estimate holds.
\end{proof}

\begin{remark}\label{rem:order2}
If the angle filtering algorithm is done first, then theorem \ref{thm:condbd} still holds.
In this case, the hypotheses of lemma 
\ref{lem:Rinv} are satisfied, and by \eqref{eqn:bound1} and proposition
\ref{prop:colbd}, the length filtering algorithm 
\ref{alg:lfilt} produces an output matrix with $\cond(F)_F < \bar \kappa$.
To update the $QR$ after length filtering, the QR factors can simply be truncated to 
agree with the updated $F$. In particular
\[
R = \begin{pmatrix} R_{11} & R_{12} \\ & R_{22} \end{pmatrix} 
~\text{ and }~ Q = \begin{pmatrix} Q_1 & Q_2 \end{pmatrix}
\] 
where $R_{11}$ and $Q_1$ correspond to the first $k$ columns of $F$ which are kept by
length filtering algorithm \ref{alg:lfilt}. Then it suffices to replace $R$ by $R_{11}$ 
and $Q$ by $Q_1$ (or $Q^\ast w_{k+1}$ by the 
corresponding entries of $Q_1^\ast w_{k+1}$, if $Q$ itself is not stored).
\end{remark}

Although the two filtering steps could be run in either order, we run the
length filtering first because it has lower computational complexity and potentially
reduces the number of columns used in the QR factorization(s) of the angle filtering
algorthm \ref{alg:afilt}. If a sharper condition estimate is preferred, the angle
filtering could be run first, after which parameter $c_s$ could be replaced 
by $c_s'$, the smallest realized value of $\sin(f_i,\cF_{i-1}) = |r_{ii}|/\|f_i\|$,
where $i$ indexes over the columns remaining after angle filtering.  Since 
$c_s' \ge c_s$, the result is a sharper estimate of the condition from length filtering.  
In our tests (not shown), this reordering did not make substantial impact  
on the total iterations to convergence.

\section{Numerics}\label{sec:numerics}
In the following numerical tests, the FAA 
algorithm \ref{alg:faa} is tested against Anderson with the TSVD applied 
to the least-squares problem as in remark \ref{rem:tsvd}.
The first three examples show the advantages of filtering in more complex applications.
The last two examples give insight into how filtering works in the preasymptotic
vs. asymptototic regimes.  
In particular, as seen in subsection \ref{subsec:quasi},
the length filtering \ref{alg:lfilt} dominates in the asymptotic regime, where the 
condition of the least-squares matrix can be compromised by the disparity in the column 
lengths as the algorithm converges.  The angle filtering \ref{alg:afilt}, on the other
hand,  stabilizes the algorithm throughout the preasymptotic regime, and may be an 
important tool in convergence from poor initial iterates.  
\subsection{Nonlinear Helmholtz equation}

We consider first the approximation of solutions to the nonlinear Helmholtz (NLH) equation from optics, where the
interest is the propagation of continuous-wave laser beams through transparent dielectrics.  The NLH system in 1D is written as: 
Find $u:[0,10]\rightarrow \mathbb{C}$ satisfying
\begin{align*}
u_{xx} + k_0^2 \left( 1 + \epsilon(x) |u|^2 \right) u  = 0, \ \ \ 
0<x<10, 
\end{align*}
where $u=u(x)$ represents the unknown (complex) scalar electric field,  
$k_0$ denotes the linear wavenumber, and $\epsilon(x)$ is a material dependent quantity
involving the linear index of refraction and the Kerr coefficient.
The physically consistent two-way boundary condition is derived in \cite{FT01,BFT09}, and is given by
\begin{align*}
u_x + i k_0 u &=2i k_0 \mbox{ at } x=0,\\
u_x - i k_0 u & =0 \mbox{ at } x=10.    
\end{align*}
Despite being 1D, this system can be quite challenging for nonlinear solvers due to its cubic nonlinearity \cite{BFT07,BFT09}.

\begin{figure}[ht]
\center
\includegraphics[width = .75\textwidth, height=.28\textwidth,viewport=0 0 1200 400, clip]{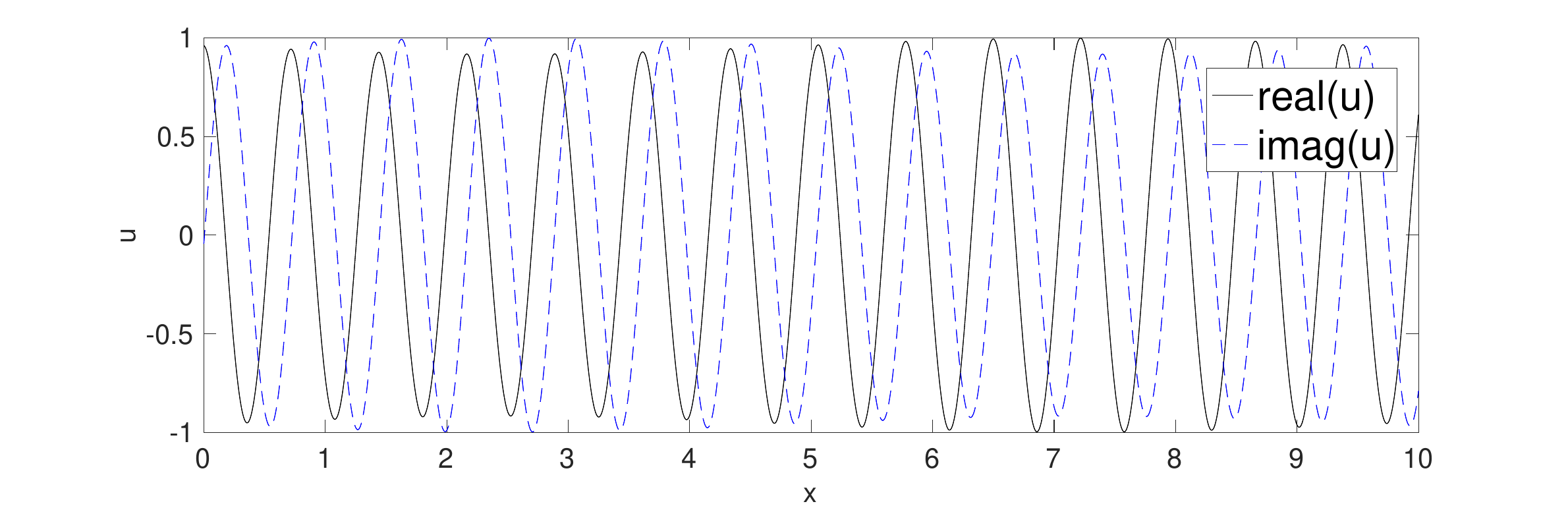}
\caption{\label{solnplotNLH} The plot above shows the solution of the nonlinear Helmholtz problem with $k_0=8$ and $\epsilon=0.2$.}
\end{figure}

A Picard-type iteration for this NLH system takes the form
\begin{align*}
{u_{j+1}}_{xx} + k_0^2 u_{j+1} + k_0^2 \epsilon(x) |  u_j |^2  u_{j+1} & = 0, \ \ \ \ \ \  0<x<10, \\
{u_{j+1}}_x + i k_0 u_{j+1} & = 2i k_0, \ \ \ x=0, \\
{u_{j+1}}_x - i k_0 u_{j+1} & = 0, \ \ \ \ \ \ \ x=10. 
\end{align*}
This iteration is advantageous due to its simplicity, and can work directly with complex numbers.  The Newton iteration, on the other hand, seems
to require decomposition into the real and imaginary parts, leading to fully coupled block linear systems at each iteration.  Here, the system is discretized
in space with a second order finite difference approximation using $N$=2001 equally spaced points, and for an initial guess we use the nodal interpolant 
of $(\cos (k_0 x) + i\sin (k_0 x))$, which is the linear ($\epsilon=0$) Helmholtz equation solution for the same $k_0$.

For our tests, we consider parameters $k_0=8$ and constant $\epsilon=0.2$, and a plot of the solution is shown in figure \ref{solnplotNLH}.  
Our computations use no relaxation ($\beta=1$) and depth $m=20$, and we compare both the condition number and the convergence of usual AA to FAA ($\bar{\kappa}=10^8$ and varying $c_s$=0.1, 0.2) to AA with TSVD (varying tolerance $10^3$ and $10^8$).  From figure \ref{convplotsNLH} at left we observe the convergence of each method, and it is clear that TSVD and FAA both have much better convergence than usual AA.  We also observe that filtering with $c_s=0.1$ provides the fastest convergence, and TSVD with both $10^3$ and $10^8$ tolerances and FAA with $c_s=0.2$ all have similar convergence behavior.  Somewhat larger and smaller $c_s$ (0.05, 0.3, 0.4) were also tested, with results similar to that of $c_s=0.2$.  We observe from figure \ref{convplotsNLH} at right that while the optimization problem of usual AA  has a high condition number near the end of the iteration, the other methods behave as designed: TSVD holds the condition number below its tolerance, and FAA holds the condition number lower than TSVD with tolerance $10^3$.

\begin{figure}[ht]
\center
\includegraphics[width = .48\textwidth, height=.35\textwidth,viewport=0 0 550 405, clip]{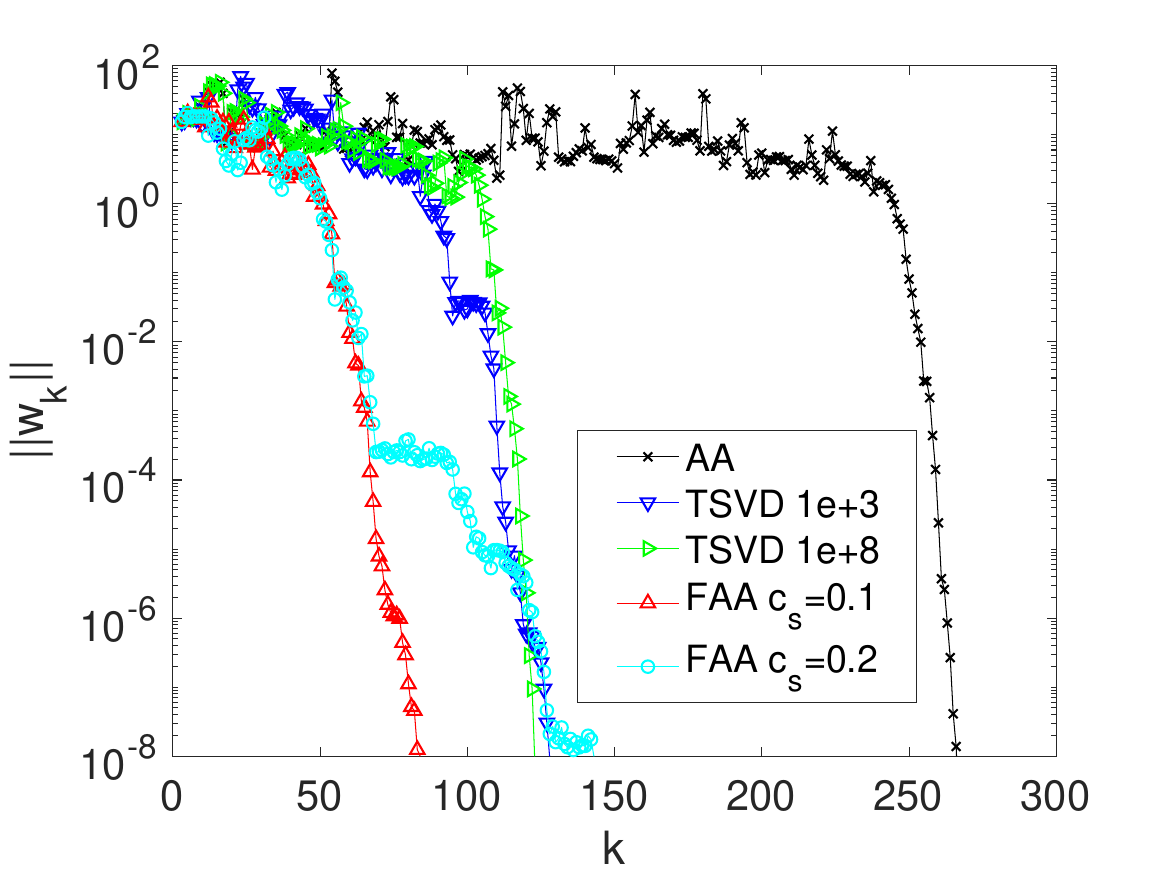}
\includegraphics[width = .48\textwidth, height=.35\textwidth,viewport=0 0 550 405, clip]{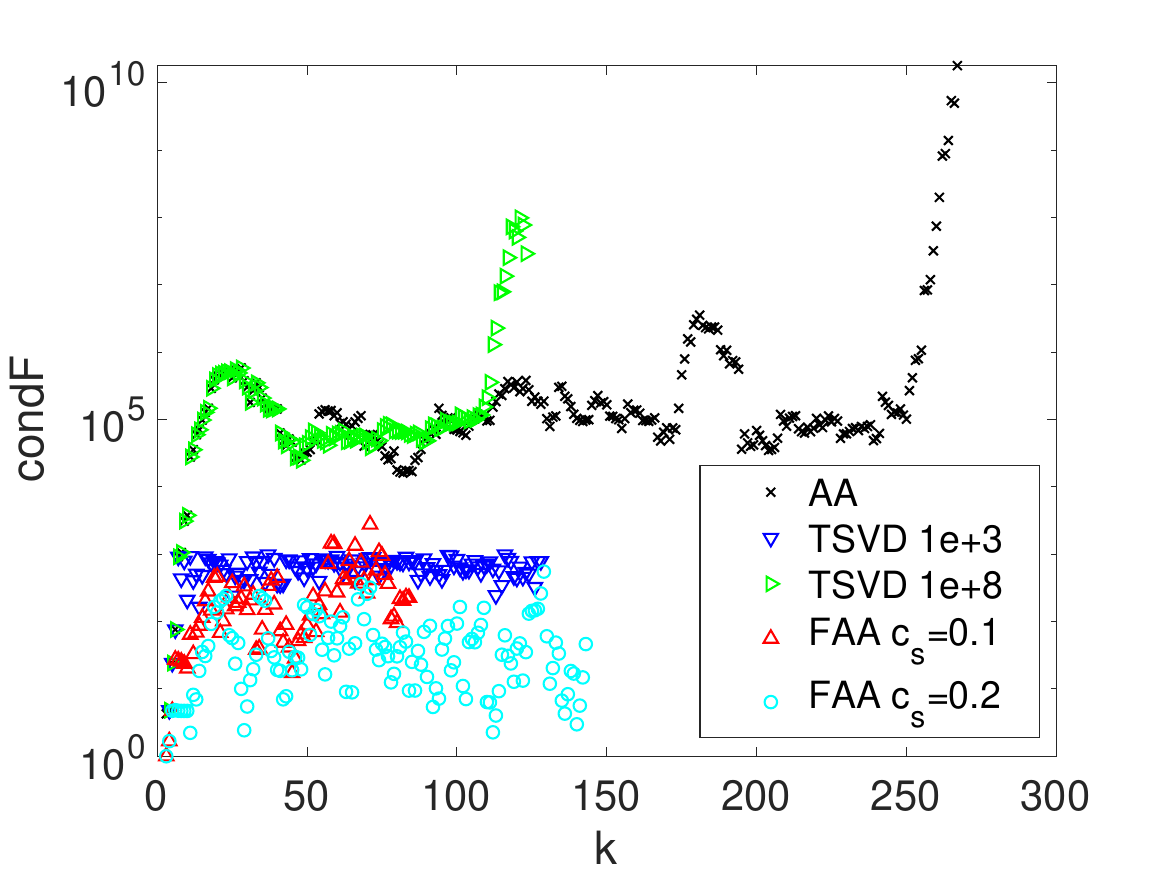}
\caption{\label{convplotsNLH} Shown above are convergence plots (left) and condition number plots (right), for AA with $m=20$ applied to the nonlinear Helmholtz problem with varying parameters for TSVD and FAA.}
\end{figure}

\subsection{2D Navier-Stokes equations}
\begin{figure}[ht]
\center
\includegraphics[width = .35\textwidth, height=.35\textwidth,viewport=115 45 465 390, clip]{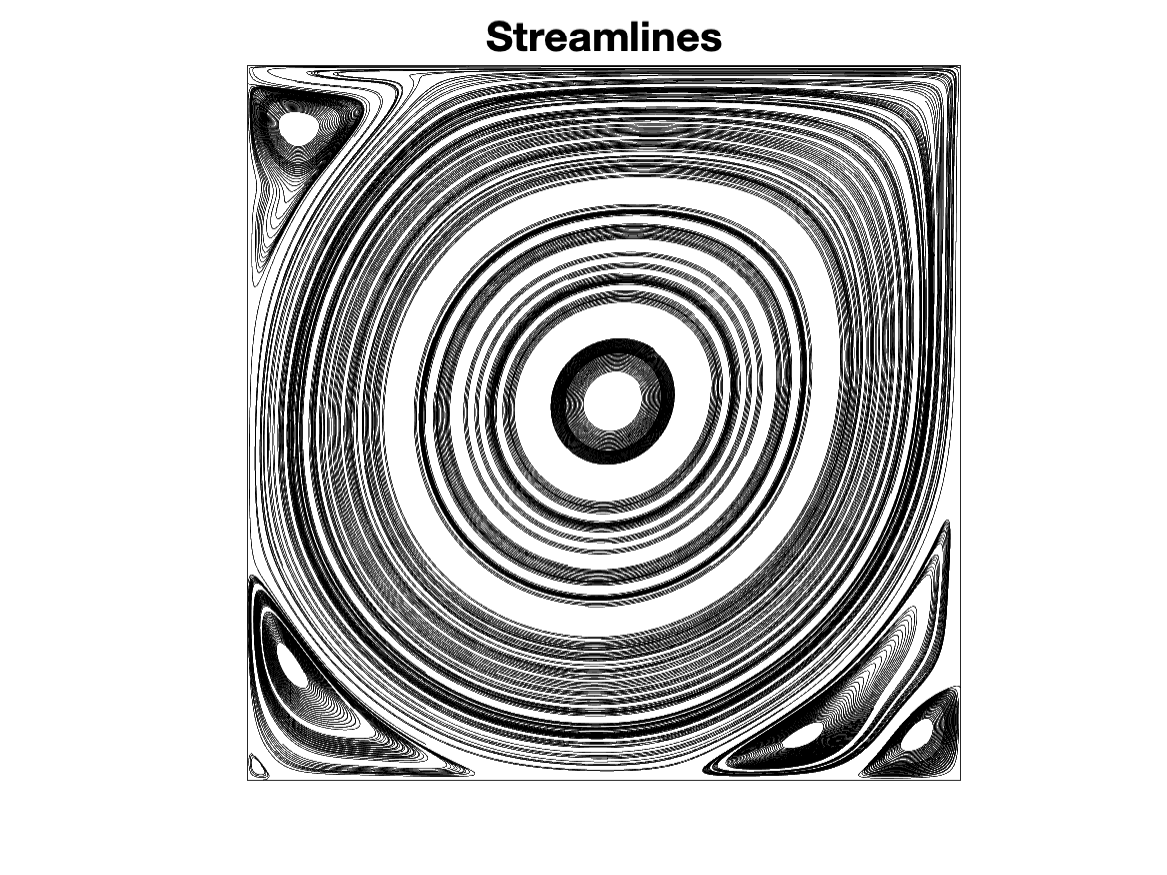} \ \ \ \ \ 
\includegraphics[width = .35\textwidth, height=.35\textwidth,viewport=100 10 500 390, clip]{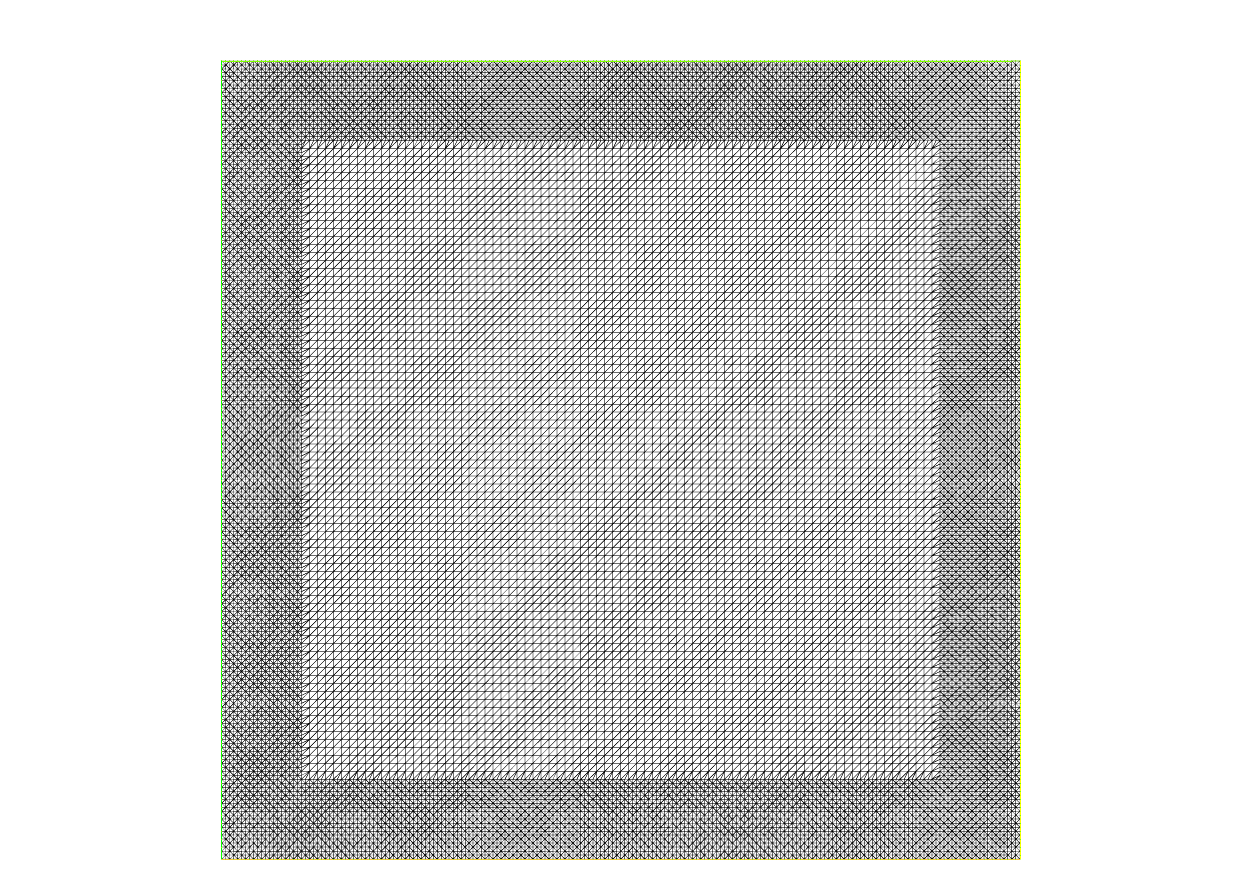}
\caption{\label{nseplot1} The plot above shows the solution of the steady Navier-Stokes 2D driven cavity problem at $Re=$10,000 (left) and the mesh used in our numerical tests (right).}
\end{figure}

Our next test is for the 2D driven cavity problem for the steady Navier-Stokes equations, which are given on a domain $\Omega=(0,1)^2$ by
\begin{eqnarray*}
u\cdot\nabla u + \nabla p - Re^{-1}\Delta u & = f, \\
\nabla \cdot u &=0,\\
u|_{\partial\Omega}&=u_{bc},
\end{eqnarray*}
where $u$ and $p$ are the unknown velocity and pressure, $f$ is an external forcing, $u_{bc}$ is a given (Dirichlet) boundary condition that is 0 on the sides and bottom and $\langle 1,0\rangle^T$ on the lid, and $Re$ is the Reynolds number.

\begin{figure}[ht]
\center
\includegraphics[width = .48\textwidth, height=.35\textwidth,viewport=0 0 550 405, clip]{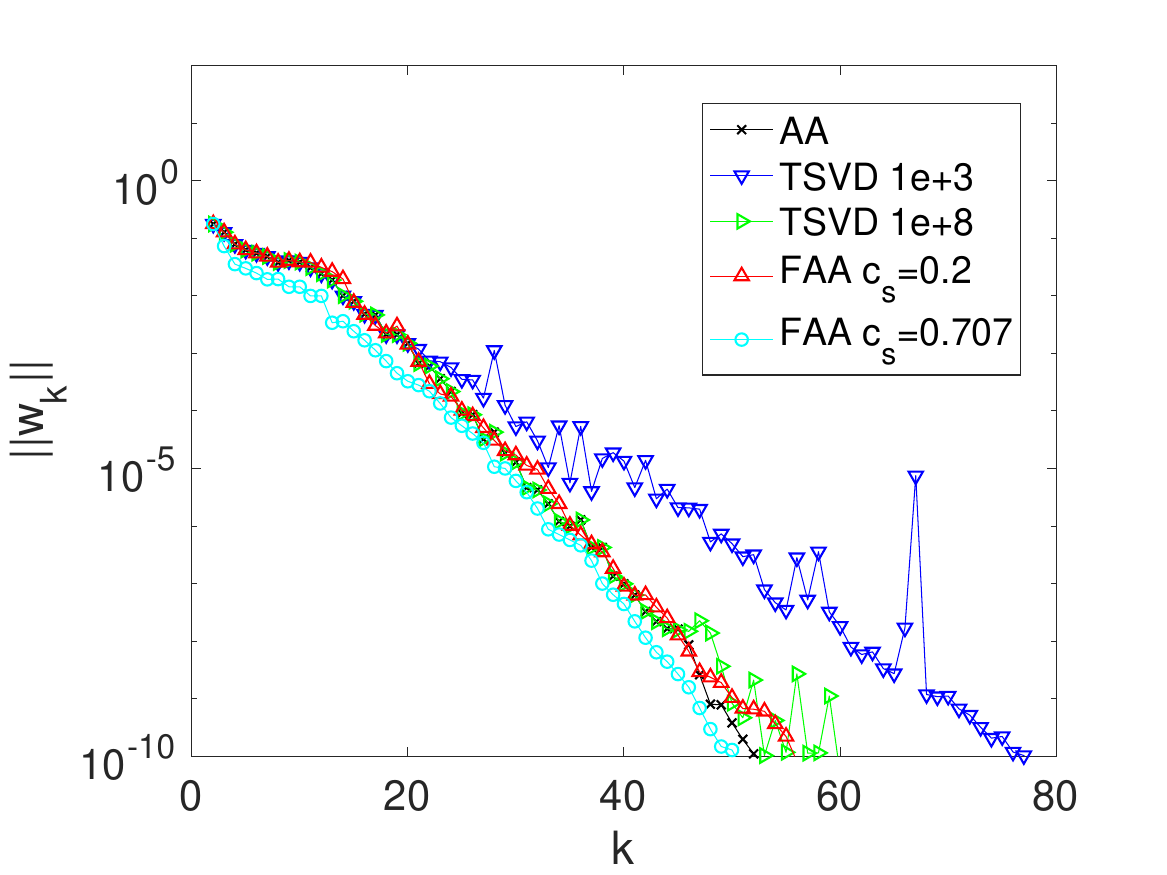}
\includegraphics[width = .48\textwidth, height=.35\textwidth,viewport=0 0 550 405, clip]{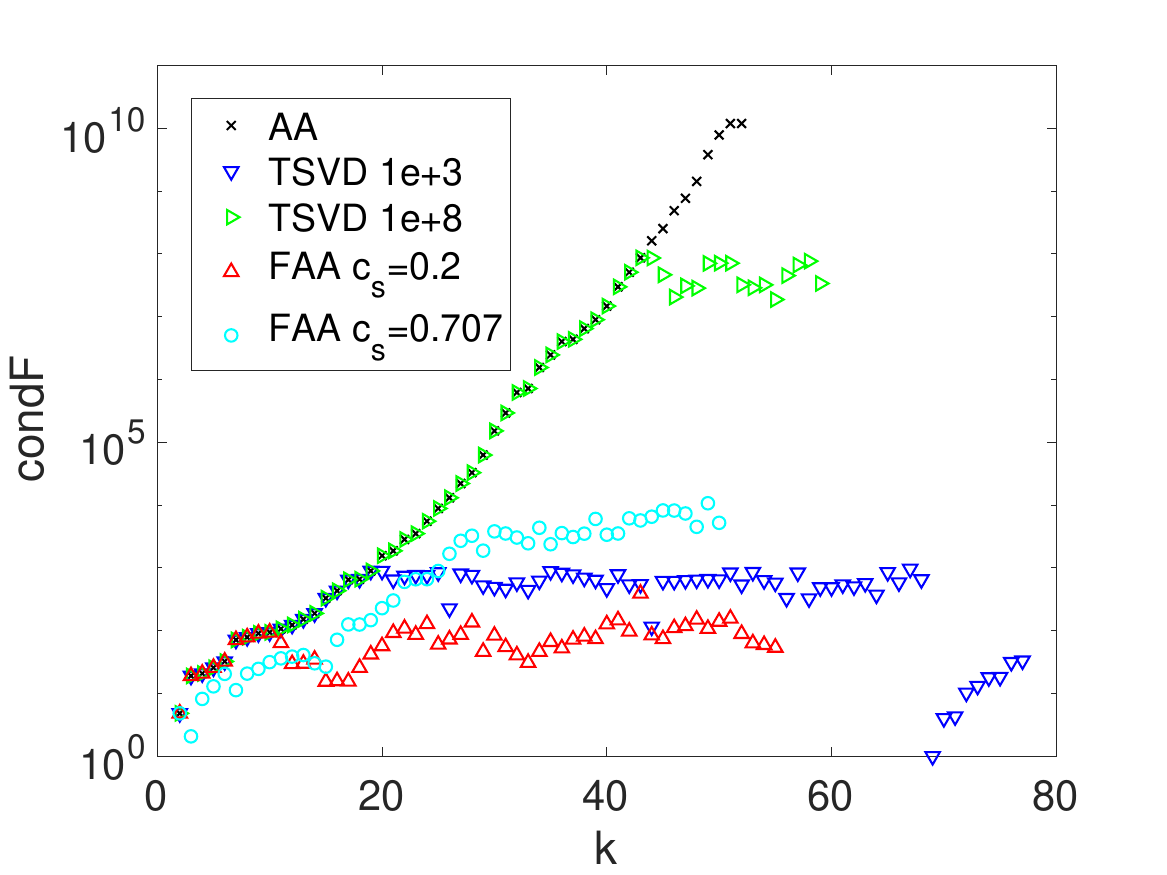}
\caption{\label{convplotsNSE} Shown above are convergence plots (left) and condition number plots (right), for 2D Navier-Stokes driven cavity test problem with varying parameters for TSVD and filtering.}
\end{figure}

The Picard iteration for the steady NSE is given by (suppressing the spatial discretization)
\begin{eqnarray*}
u_{j} \cdot\nabla u_{j+1} + \nabla p_{j+1} - Re^{-1}\Delta u_{j+1} & = f, \\
\nabla \cdot u_{j+1} &=0,\\
u_{j+1}|_{\partial\Omega}&=u_{bc},
\end{eqnarray*}
and we use an initial guess of zero $u_0=0$ (which implies that $u_1$ is the Stokes solution with this problem data).  The iteration is enhanced with $m=50$ AA and no relaxation ($\beta=1$).  As found in \cite{PR21,RVX21}, large $m$ (even $m=k-1$) for steady NSE problems tends to work well and often better than smaller $m$ when $Re$ is large.  We note that for this test, $m=0$ gave no convergence, $m=1$ gave slow convergence, and $m\ge 10$ gave convergence similar to $m=50$ (tests omitted).  A key advantage of the filtering strategy proposed herein is that one can choose large $m$ and if $m$ is ``too large'' on a particular iteration (large condition number of $F$) then $m$ 
is effectively
reduced in an optimal manner (with respect to the optimization problem).

The spatial discretization is grad-div stabilized  (with parameter 1) $(P_2,P_1)$ Taylor Hood elements on a triangulation created from a uniform $100 \times 100$ triangulation that is further refined once around the edges (within 0.1 from the boundary), which provides a total of 190,643 degrees of freedom; a plot of the mesh is shown in figure \ref{nseplot1} at right.

Results for convergence and condition number versus iteration are given in figure \ref{convplotsNSE} for usual AA, filtering with $\bar{\kappa}=10^8$ and $c_s=0.2$ and $1/\sqrt{2}$, and TSVD with tolerances $10^3$ and $10^8$.  
We observe that usual AA has a very large condition number near the end of the iteration, and that both TSVD and FAA can control this.  
We observe good convergence with all methods except TSVD with tolerance $10^3$ displays some erratic behavior and converges somewhat slower than the other methods.  Only filtering with $c_s=1/\sqrt{2}$ beat usual AA, although only slightly.

\subsection{Gross-Pitaevskii equations}\label{subsec:gs}

For our next test, we consider the Gross-Pitaevskii equations (GPE).  
In addition to demonstrating how well FAA works on another important 
application problem, this test will show how filtering naturally works in conjunction 
with the dynamic depth techniques of \cite{PR21}.  The GPE are given by
\begin{align*}
\mu \phi(x) = -\frac12 \Delta \phi(x) &+ V(x) \phi(x) + \eta |\phi(x)|^2 \phi(x), \ \ x\in \Omega,  \\
\phi(x) & = 0, \ \ x \in \partial\Omega,  \\
\int_{\Omega} |\phi(x)|^2\ dx & = 1, 
\end{align*}
where $V$ is a given trapping potential of the form $V(x)=\frac12 (\gamma_1^2 x_1^2 + ... + \gamma_d^2 x_d^2)$ with $\gamma_i>0 \ \forall i$, real parameter $\eta$,
and $\phi$ is the unknown and the eigenvalue $\mu$ can be calculated as
\begin{equation*}
\mu= \int_{\Omega} \left( \frac12 | \nabla \phi |^2 + V|\phi|^2 + \eta | \phi |^4 \right)\ dx. 
\end{equation*}
This system models stationary solutions of the nonlinear Schr\"odinger (NLS) equation, which is also commonly referred to
as the the non-rotational GPE in the context of Bose-Einstein condensates (BEC) \cite{LL77,P61,BD04}.   In the 
GPE setting, $\phi$ represents the macroscopic wave function of the condensate and the parameter $\eta$ being positive/negative represents attraction/repulsion
of the condensate atoms in GPE.   

We apply AA with and without TSVD, and FAA to a FEM discretization of the Picard-projection iteration from \cite{FRX21}, which we call PP$_{h}$ and is defined by:\\
\ \\
PP$_{h}$ Step 1: Given $\phi_k \in X_h$ with $\| \phi_k\|=1$, find $\hat \phi_{k+1}\in X_h$ satisfying 
\[
\frac12 (\nabla  \hat \phi_{k+1}, \nabla \chi) + (V \hat \phi_{k+1},\chi) + \eta (|\phi_k|^2 \hat \phi_{k+1},\chi)  = (\phi_k,\chi).
\]
PP$_{h}$  Step 2: Calculate
\[
\phi_{k+1}  = \frac{\hat \phi_{k+1} }{ \| \hat \phi_{k+1} \|}.
\]
\ \\ \ \\
We consider PP$_h$ as a fixed point iteration solver for a 2D test from \cite{BD04} that represents a harmonic oscillator potential together with a potential from a stirrer that corresponds to a far-blue detuned Gaussian laser beam, i.e.
\[
V(x,y)=\frac12 \left( x^2 +  y^2\right) + 4 e^{-\left( (x-1)^2 + y^2\right)}.
\]
Solutions are computed on $\Omega=(-8,8)^2$ using $\eta=$10,000 and initial guess 
\[
\phi_0(x,y)=\frac{1}{\pi^{1/2}} e^{-(x^2 + y^2)/2 },
\]
using a uniform 128x128 triangulation and globally continuous $P_2$ elements.  A plot of the solution is shown in figure \ref{AAPPplots} as a surface plot, which agrees well with the literature \cite{FRX21,BD04}.

\begin{figure}[h!]
\begin{center}
\includegraphics[width = .32\textwidth, height=.3\textwidth,viewport=0 0 550 400, clip]{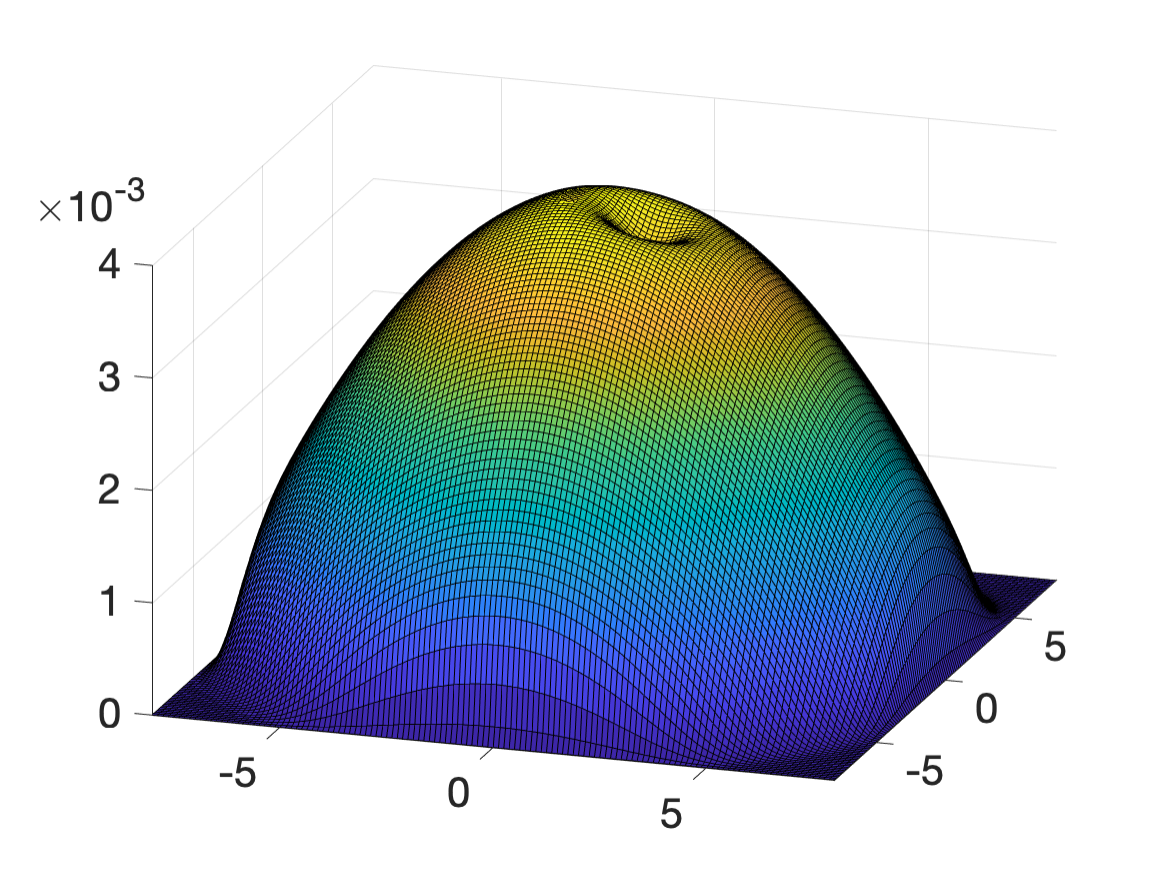}
\caption{\label{AAPPplots} Shown above is a surface plots of the converged ground state solution using $\eta=$10,000.}
\end{center}
\end{figure}

PP$_h$ is run with AA $m=1$ until the residual is below $10^{-2}$, after which $m=20$; 
no relaxation is used.  This multilevel depth strategy of small $m$ early and large $m$ late in the iteration proposed in \cite{PR21} (and motivated by their analytical convergence analysis) was found effective in \cite{FRX21} for this test problem, and in fact we find that if $m \ge 10$ is used for the entire iteration then we do not get convergence (a plot of failed convergence for constant m=20 AA is shown in figure \ref{gpe1}, and we note that AA with constant $m=20$ neither TSVD nor filtering with the same parameters as below provided for convergence).  

Tests are run with multilevel $m=1$ then $m=20$ AA with no filtering, FAA with $\bar\kappa=10^8$ and $c_s=$0.2:0.1:0.7, and TSVD with tolerances $10^3$ and $10^8$.  Plots of convergence and condition numbers for these tests are shown in figure \ref{gpe1}, and we observe that usual AA, TSVD with tolerances $10^8$ and $10^3$, and FAA with $c_s=$0.2 and 0.6 all had similar good convergence ($c_s$=0.3, 0.4 and 0.5 converged similarly as well; plots omitted).  FAA with $c_s=$0.7 converged significantly slower, and upon inspection we found that this filtering effectively reduced the AA to $m=1$ at each step.  AA with constant $m=20$ failed to converge, with or without filtering and TSVD (at least, with the same parameters as above).

\begin{figure}[ht]
\center
\includegraphics[width = .48\textwidth, height=.35\textwidth,viewport=0 0 850 605, clip]{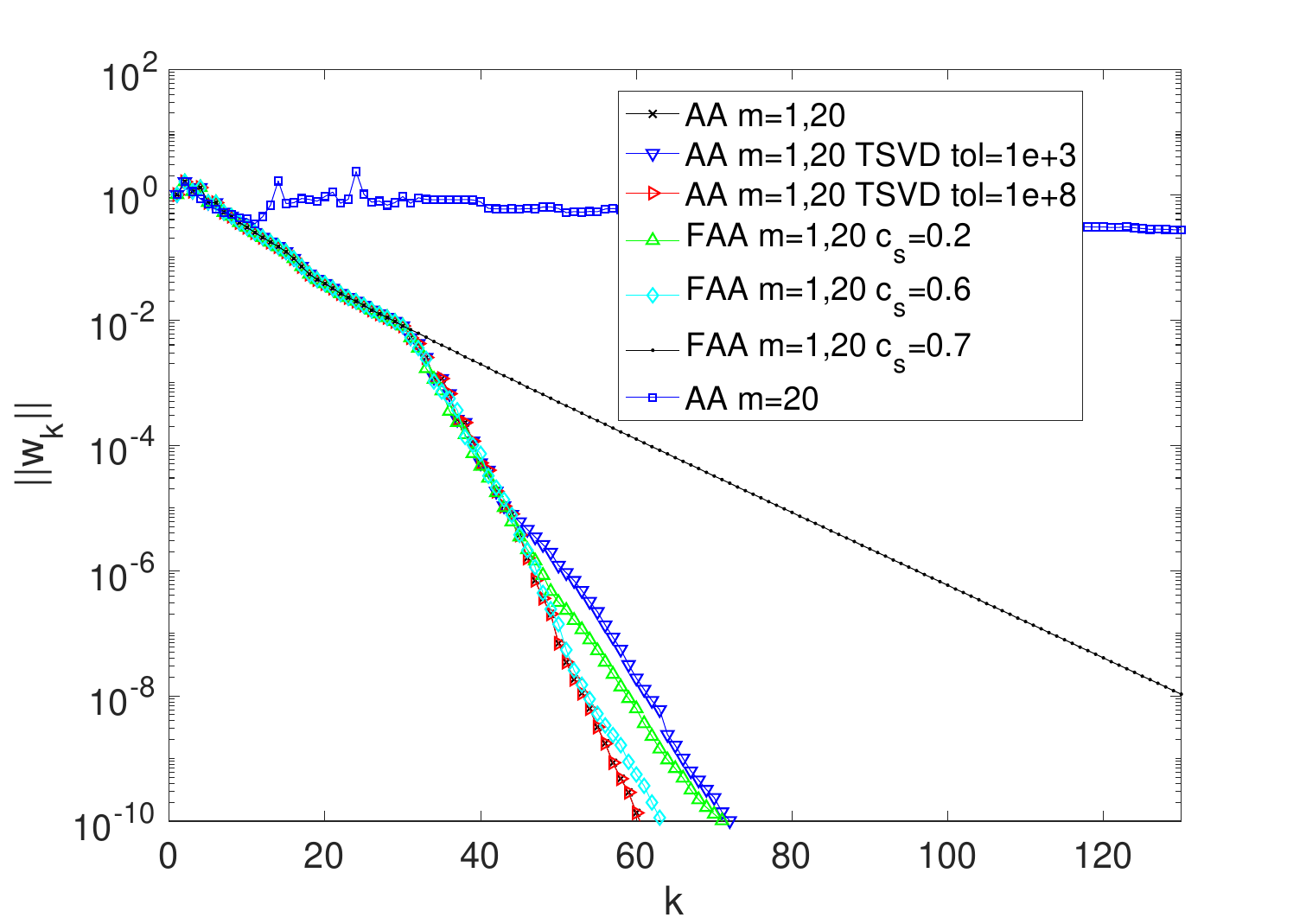}
\includegraphics[width = .48\textwidth, height=.35\textwidth,viewport=0 0 850 605, clip]{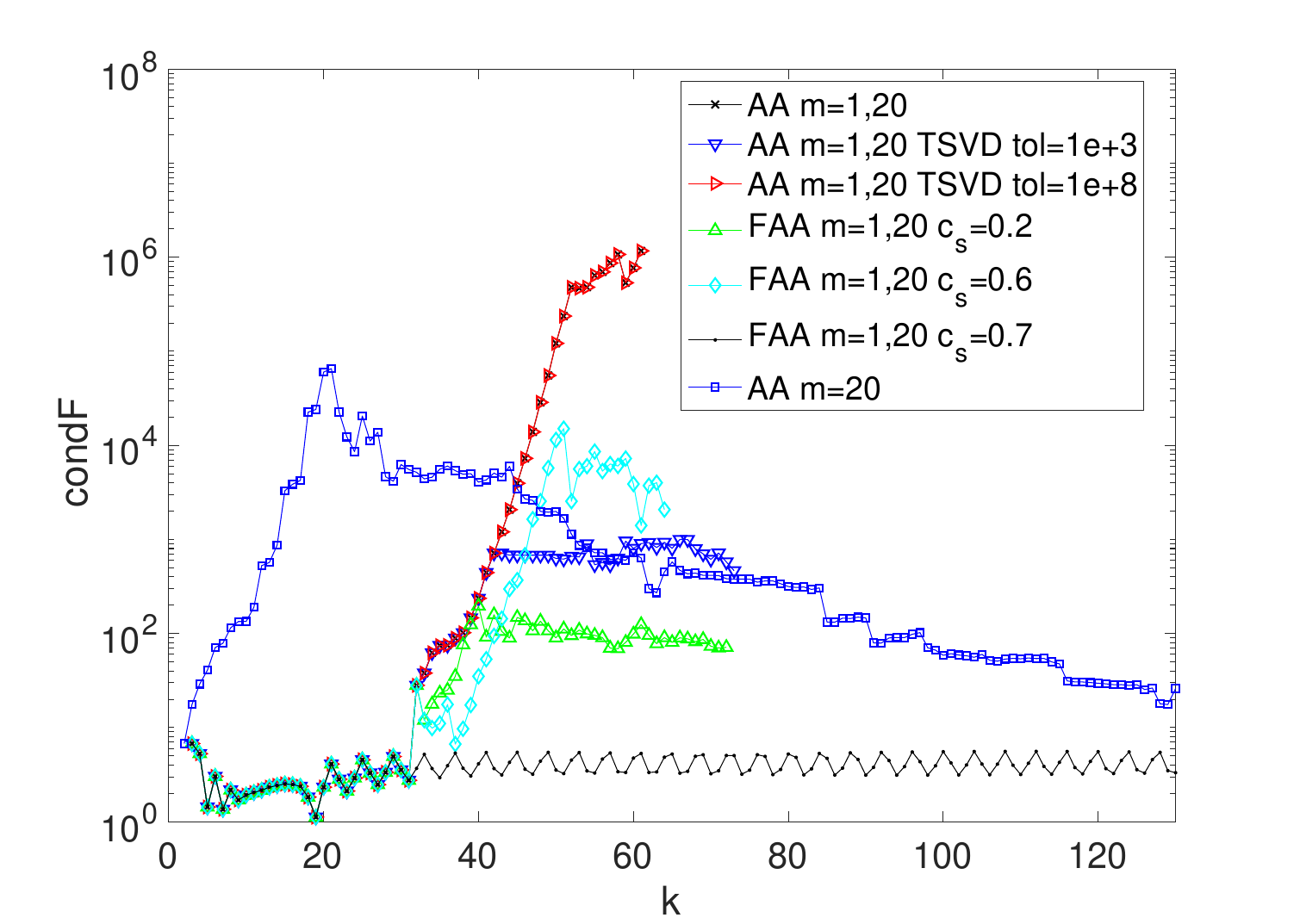}
\caption{\label{gpe1} Shown above are convergence plots (left) and condition number plots (right), for the Gross-Pitaevskii test problem.}
\end{figure}

\subsection{A monotone quasilinear equation}\label{subsec:quasi}
This example of an iteration for a finite element discretization of
monotone quasilinear problems, adapted from \cite[Example 1]{CoWi17}, 
illustrates the convergence behavior of filtering in the asymptotic regime.
We consider the PDE 
\begin{align}\label{eqn:monoql}
-\divi(\mu(|\grad u|)\grad u) = f, ~\text{ in }
\Omega = (0,1) \times (0,1), \quad u = 0 \text{ on } \pa \Omega,\\
\text{ with }
\quad \mu(|\grad u|) = 1 + \arctan(|\grad u|),
\quad f = \pi.
\end{align}
In our experiments the solution space $V_h$ consists of $C^0$ piecewise quadratic ($P_2$)
finite elements on
a uniform triangularization of the unit square with right triangles and 256
subdivisions along each of the $x$ and $y$ axes for a total of 263,169 degrees of 
freedom.
Each update step $w_{k+1}$ is found by the solution to the linear problem:
Find $w_{k+1} \in V_h$ such that
\begin{align}\label{eqn:qlupdate}
(\grad w_{k+1}, \grad v) = (f,v) - (\mu(|\grad u_k|)\grad u_k,\grad v), 
\text{ for all } v \in V_h.
\end{align}
By the analysis of \cite{CoWi17}, the fixed-point iteration 
$g(u_k) = u_k + \beta w_{k+1}$ is contractive for 
$0 < \beta \le \beta^\ast \coloneqq (1 + \sqrt 3/2 + \pi/3)^{-2}$.
Accelerating the fixed-point iteration written in terms of a constant damping factor 
$\beta$ is the same as accelerating the original fixed-point iteration 
$g(u_k) = u_k + w_{k+1}$ using $\beta_k = \beta$ for each iteration $k$ in either
original AA or FAA algorithm \ref{alg:faa}.
We consider iterations both with and without the damping factor $\beta$ sufficiently
small.
In table \ref{tabl:qmono-tsvd} we show the number of iterations to residual convergence 
$\nr{w_k} < 10^{-10}$ from the starting guess of $u_0 = 0$.   
We compare the filtering algorithm with maximum condition number $\bar \kappa = 10^8$
with the TSVD algorithm with maximum condition number up to $10^8$. 

\begin{table}
\begin{center}
{\footnotesize
\begin{tabular}{|c|r|r|r|r|r|r|r|r|}
\hline
& \multicolumn{4}{c|}{$\beta = \beta^\ast$} & \multicolumn{4}{c|}{$\beta = 1$} \\
\hline
FAA, $\bar \kappa = 10^8$ 
& $m=5$ & $m=10$ & $m=20$ & $m=40$ & $m=5$ & $m=10$ & $m=20$ & $m=40$ \\
$c_s = 0.1$  & 32 & 27 & 27 & 27 & 21 & 20 & 20 & 20 \\
$c_s = 0.4$  & 31 & 31 & 31 & 31 & 21 & 21 & 21 & 21 \\
$c_s= 2^{-1/2}$& 96 & 96 & 96 & 96 & 22 & 23 & 23 & 23 \\
\hline
TSVD & $m=5$ & $m=10$  & $m=20$ & $m=40$ & $m=5$ & $m=10$ & $m=20$ & $m=40$ \\
$\bar \kappa = 10^2$   & 42 & 57 & 77 & 100 & 45 & 64 & 119 & 238 \\
$\bar \kappa = 10^3$  & 35 & 45 & 68 & 96  & 32 & 38 & 80  & 164 \\
$\bar \kappa = 10^4$   & 30 & 40 & 64 & 91  & 30 & 33 & 63  & 121 \\
$\bar \kappa = 10^6$  & 33 & 36 & 58 & 96  & 30 & 26 & 45  & 85  \\
$\bar \kappa = 10^8$   & 33 & 38 & 64 & 92  & 30 & 22 & 35  & 51  \\
\hline
\end{tabular}
}
\end{center}
\caption{The number of iterations to residual convergence $\nr{w_k} < 10^{-10}$ for
the monotone quasilinear problem \eqref{eqn:monoql} with iteration 
\eqref{eqn:qlupdate}, using 
FAA with $c_s = 0.1, 0.4, 2^{-1/2}$, and
the TSVD with different maximum condition numbers $\bar \kappa$ and $m=5,10,20,40$. 
Without acceleration, the iteration with
$\beta = \beta^\ast$ converges in 175 iterations, and the iteration with $\beta = 1$
does not converge.
}
\label{tabl:qmono-tsvd}
\end{table}

For the contractive iteration with $\beta = \beta^\ast$ we see the 
TSVD approach for limiting the condition number overall slows convergence as the 
maximum allowed condition number $\bar \kappa$ is decreased. 
The filtering approach also limits the condition number but 
improves convergence, so long as the minimum allowable angle $c_s$ between columns of 
the least squares matrix $F_k$ in the angle filtering algorithm \ref{alg:lfilt} is 
not too large.  In this case $c_s = 2^{-1/2}$ is too large,
and restricts the algorithmic depth after filtering to $m=1$. 
Results for the FAA algorithm \ref{alg:faa} are only shown here for the 
maximum condition number $\bar \kappa$ shown for the TSVD since the length filtering
algorithm \ref{alg:lfilt} overestimates the condition, hence as shown in the 
previous examples actually restricts the condition number to several orders of magnitude
less than $\bar \kappa$. The agreement in the iteration counts for filtering as 
algorithmic depth $m$ is increased reflects the length-filtering algorithm cutting 
off older columns from $F_k$ so that increasing $m$ does not actually change the 
iteration after a certain point, whereas the TSVD continues to use all columns available
which can slow convergence. 

Filtering in the contractive setting is further illustrated with figure 
\ref{fig:qlbetamoncols}, which shows the 
columns of $F_k$ used at each iteration $k$ with the filtering approach with 
algorithmic depth $m=10$. The left two plots illustrate the columns of $F_k$ used
for the data shown in 
table \ref{tabl:qmono-tsvd} with $\bar \kappa = 10^8$ and  
$c_s = 0.1$ (leftmost) and $c_s = 0.4$ (to its right). The case
with $c_s = 2^{-1/2}$ is not shown because it simply restricts the iteration to 
$m=1$ throughout the entire iteration.  The right two plots in figure 
\ref{fig:qlbetamoncols} show the columns of $F_k$ used for 
$c_s = 0.1$ and $c_s = 0.4$ (rightmost) with $\bar \kappa
= 10^{16}$.  In this case we see for $c_s = 0.1$ the difference between the lower
and higher condition bound is that the length filtering in the second half of the 
iteration is surpressed for $\bar \kappa = 10^{16}$, although this does not have much 
of an effect on the total number of iterations.  For $c_s = 0.4$, the angle-filtering, 
which filters out earlier columns, dominates throughout both iterations.  

\begin{figure}
\centering
\includegraphics[trim = 310pt 230pt 322pt 240pt,clip = true,scale = 0.55]
{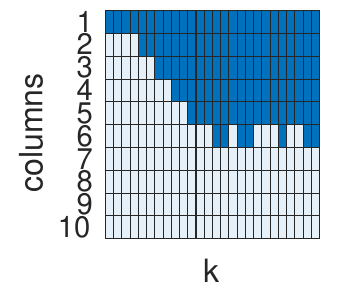}
\includegraphics[trim = 310pt 230pt 322pt 240pt,clip = true, scale = 0.55]
{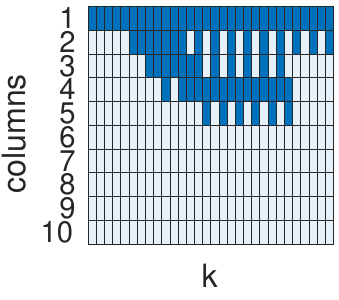}
\includegraphics[trim = 310pt 230pt 322pt 240pt,clip = true, scale= 0.55]
{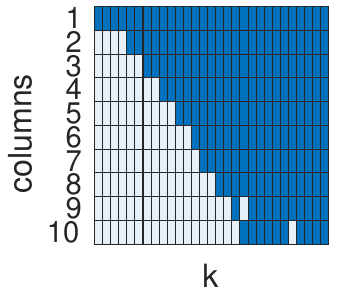}
\includegraphics[trim = 310pt 230pt 322pt 240pt,clip = true, scale = 0.55]
{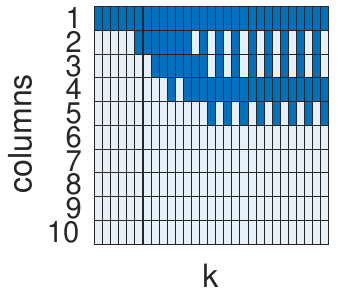}
\caption{The columns used after filtering for
the monotone quasilinear problem \eqref{eqn:monoql} with iteration 
\eqref{eqn:qlupdate} with $m=10$ and $\beta = \beta^\ast$. 
Column 1 is always used and column 2 may be filtered out in angle filtering.
From left to right:  
$c_s = 0.1, \bar \kappa = 10^8$; 
\ $c_s = 0.4, \bar \kappa = 10^8$ (as in table \ref{tabl:qmono-tsvd}); 
\ $c_s = 0.1, \bar \kappa = 10^{16}$; 
\ $c_s = 0.4, \bar \kappa = 10^{16}$. 
}
\label{fig:qlbetamoncols}
\end{figure}

For damping factor $\beta =1$, the iteration defined by update step \eqref{eqn:qlupdate}
is not contractive, and the iteration does not converge without acceleration.
It is interesting to compare the iteration counts for $c_s = 2^{-1/2}$ for 
$\beta = \beta^\ast$ and $\beta =1$. For $\beta =1$, the stricter requirement on the 
angles between columns of $F_k$ does not make a large impact on convergence. As seen
in the rightmost plot of figure \ref{fig:qlbeta1cols}, only the second column is 
regularly filtered out.  This illustrates that smaller damping factors can actually 
cause alignment of columns of $F_k$, which can interfere with both convergence and 
the condition of the least-squares problem. Comparing with the TSVD in table
\ref{tabl:qmono-tsvd}, the filtered algorithm converges after fewer total 
iterations, and is robust across parameter ranges tested, whereas the TSVD shows 
similar performance only when the right number of columns $m=10$ is used, and when the
allowable condition number is high enough. The second interesting observation that
we see in both figures \ref{fig:qlbetamoncols} and \ref{fig:qlbeta1cols} is the
angle filtering algorithm \ref{alg:afilt} tends to filter out the earlier (leftmost)
rather than the later columns, and most often the second.  We see this again in the next
example for a problem where the approximation progresses through a more substantial
preasymptotic regime, and where a more stringent condition on the angle between columns
of $F_k$ can be very useful.

\begin{figure}
\centering
\includegraphics[trim = 320pt 230pt 320pt 240pt,clip = true,scale = 0.70]
{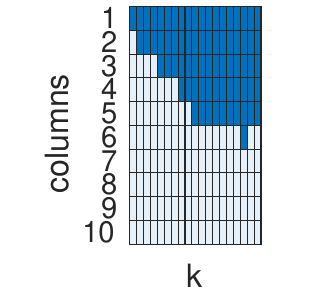}
\includegraphics[trim = 320pt 230pt 320pt 240pt,clip = true, scale = 0.70]
{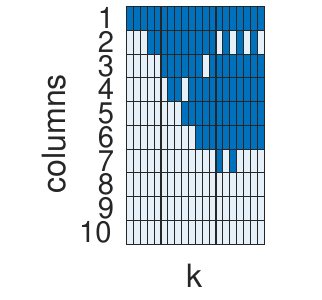}
\includegraphics[trim = 320pt 230pt 320pt 240pt,clip = true, scale= 0.70]
{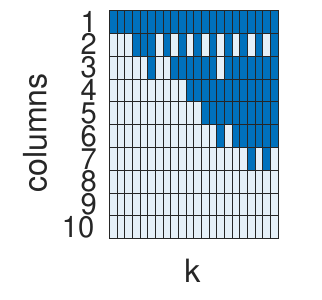}
\caption{The columns used after filtering for
the monotone quasilinear problem \eqref{eqn:monoql} with iteration
\eqref{eqn:qlupdate} with $m=10$, $\beta = 1$ and $\bar \kappa = 10^8$, 
as in table \ref{tabl:qmono-tsvd}. 
Column 1 is always used and column 2 may be filtered out in angle filtering.
Left: $c_s = 0.1$,  center: $c_s = 0.4$, right: $c_s = 2^{-1/2}$.}
\label{fig:qlbeta1cols}
\end{figure}

\subsection{$p$-Laplace}\label{subsec:plap}
In this example we consider finite element discretizations of the elliptic $p$-Laplace
equation
\[
-\divi \left( (|\nabla u|^2/2)^{(p-2)/2} \grad u\right) = f, ~\text{ in } \Omega.
\]
The $p$-Laplace (or $p$-Poisson) 
equation arises frequently as a model problem in non-Newtonian and 
turbulent flows and flows in porous media \cite{DiTh94,DFTW20}, as well as in the 
discretization of abstract linear operators in Banach space \cite{MuZe20}.
If $1 < p < 2$ the equation is singular, as the coefficient on the nonlinear diffusion
coefficient is negative and blows up as $|\grad u|$ goes to zero. Hence we consider
a regularized equation
\begin{align}\label{eqn:pLap-reg}
-\divi\left( \left(\eps^2 + \frac 1 2 |\nabla u|^2\right)^{(p-2)/2} \nabla u \right) = f,
~\text{ in } \Omega ~\text{ with } ~\eps = 10^{-14}. 
\end{align}

\begin{table}
\begin{center}
\begin{tabular}{|c|r|r|r|r|r|r|r|r|r|}
\hline
FAA, $\bar \kappa = 10^8$    & $P_1$ & $P_2$ & $P_3$ & $P_4$ \\
\hline
$c_s = 0.1$   & F   & 364 & 273 & 218 \\
$c_s = 2^{-1/2}$& 198 & 171 & 203 & 243 \\
$c_s$ dynamic & 142 & 134 & 152 & 215 \\
\hline
TSVD    & $P_1$& $P_2$& $P_3$ & $P_4$ \\
\hline
$\bar \kappa = 10^2$  & $>$500 & 207    & 154  & $>$500 \\
$\bar \kappa = 10^3$  & F      & F      & $>$500 & F    \\
$\bar \kappa = 10^4$  & F      & F      & $>$500 & F    \\
$\bar \kappa = 10^6$  & F      & F      & $>$500 & F    \\
$\bar \kappa = 10^8$  & F      & F      & $>$500 & F    \\
\hline
\end{tabular}
\end{center}
\caption{The number of iterations to residual convergence $\nr{w_k} < 10^{-10}$ 
for problem \eqref{eqn:pLap-reg} using update \eqref{eqn:pLap1} discretized with 
Lagrange $P_k$ elements with $k = 1,2,3,4$. Results are shown using
the TSVD with different maximum condition numbers $\bar \kappa$, and 
FAA 
with $c_s = 0.1, 2^{-1/2}$ and a dynamically chosen $c_s$ given by \eqref{eqn:csdyn}. 
For runs that did
not converge within 500 iterations, $>500$ demarks cases where the last 10 residuals
are less than 1 suggesting the iteration might eventually converge; otherwise 
iterations failing to converge are marked F. All runs used algorithmic depth $m=10$.
}
\label{tabl:pLap-tsvd}
\end{table}

The equation for $1 < p < 2$ is more challenging to solve for $p$ closer to one, 
and here we
use $p = 1.04$, which is quite small, {\em cf.,} \cite{DFTW20,PR21}. Similarly to
\cite{BKST15,PR21}, we consider \eqref{eqn:pLap-reg} over the domain 
$\Omega = (0,2)\times(0,2)$ with homogeneous Dirichlet boundary conditions, 
constant forcing $f = \pi$, and a relatively bad initial guess 
$u_0 = xy(x-1)(y-1)(x-2)(y-2)$ for the Picard iteration given by the update step:
Find $w_{k+1}\in V_h$ satisfying for all $v\in V_h$
\begin{align} \label{eqn:pLap1}
\int_\Omega a_{\eps,p}(u_k)\grad w_{k+1}\cdot \grad v \dd x = 
\int_\Omega fv \dd x - \int_\Omega a_{\eps,p}(u_k) \grad u_k \cdot \grad v \dd x,
\end{align}
with $a_{\eps,p}(u_k) = \left(\eps^2 + \f 1 2 |\grad u_{k}|^2\right)^{(p-2)/2}$.
We consider discretizations with $C^0$ $P_k$ Lagrange elements for $k = 1,2,3,4$, 
over a uniform triangularization of $\Omega$ with right triangles with 256 
subdivisions in each of the $x$ and $y$ axes, for a total of 66,049 ($P_1$);
263,169 ($P_2$); 591,361 ($P_3$); and 1,050,625 ($P_4$) respective degrees of freedom.
We compare the FAA algorithm \ref{alg:faa} with both fixed and 
dynamically chosen parameters to the TSVD approach. The angle filtering parameters used 
in this example are $c_s = 0.1$, $c_s = 2^{-1/2}$ and the dynamic choice
\begin{align}\label{eqn:csdyn}
c_s = \max\left\{\min\left\{ \|w_{k+1}\|^{1/2},2^{-1/2}  \right\} ,0.1\right\},
\end{align}
which transitions between the smaller and larger choices of $c_s$ based on 
the norm of the 
fixed-point
update step.

Table \ref{tabl:pLap-tsvd} shows the number of iterations to residual convergence
$\nr{w_k} < 10^{-10}$ for the filtered algorithm \ref{alg:faa}. We see the filtering 
algorithm is an enabling
technology that allows us to solve this problem at all, whereas the TSVD approach is
only succssful when the condition number is stricly controlled with $\bar\kappa = 10^2$,
and even then we only see a successful solve within 500 iterations for the $P_2$ and 
$P_3$ 
elements. The filtering method is successful in all but one case: $c_s = 0.1$ with 
$P_1$ elements.

\begin{figure}
\includegraphics[trim = 0pt 0pt 40pt 20pt,clip = true,width=.32\linewidth, height=.3\linewidth]
{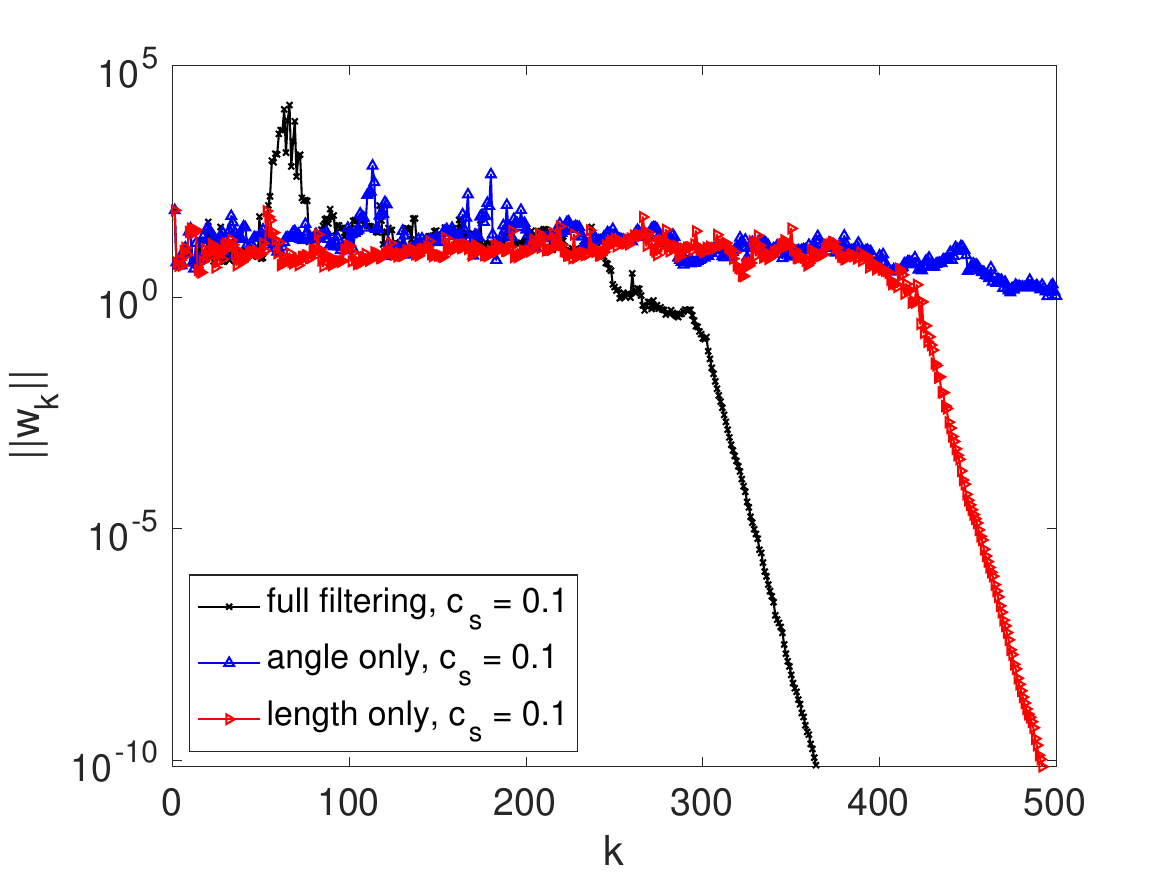}
\includegraphics[trim = 0pt 0pt 40pt 20pt,clip = true, width=.32\linewidth, height=.3\linewidth]
{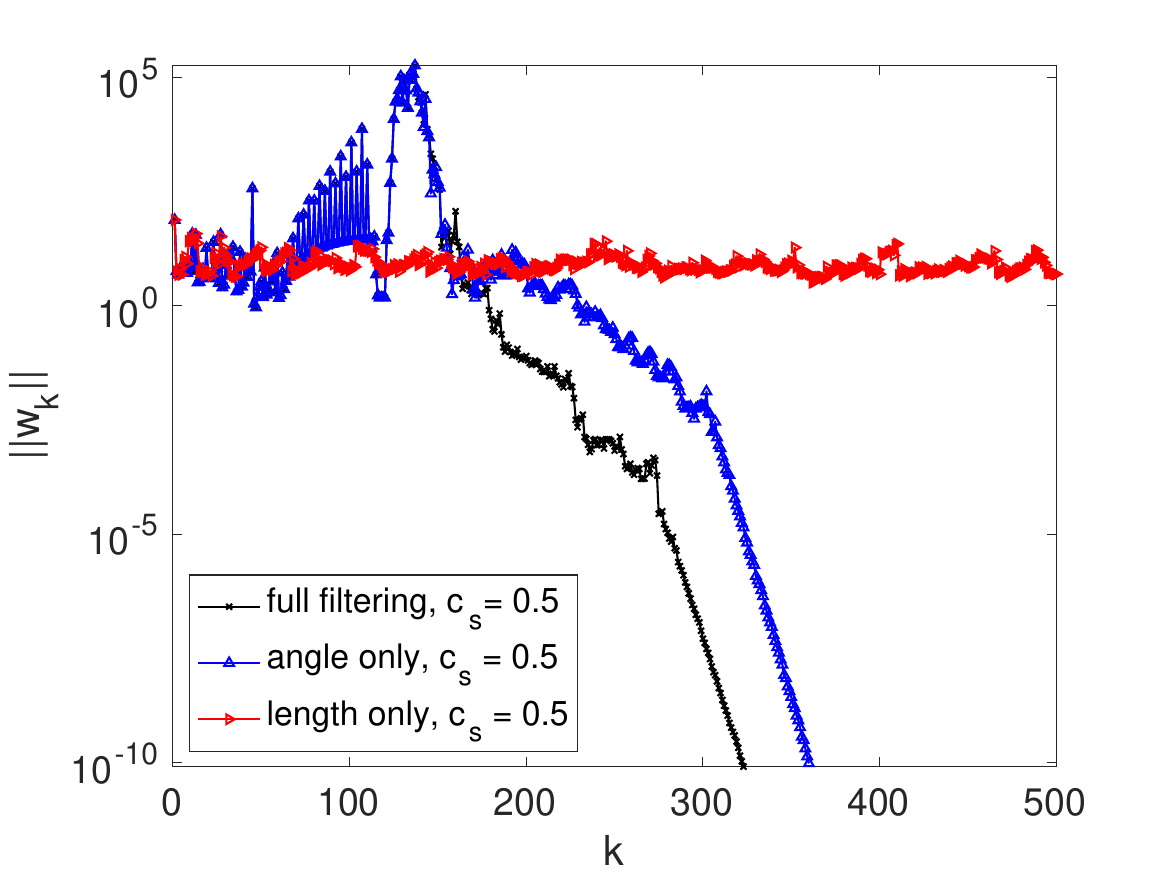}
\includegraphics[trim = 0pt 0pt 40pt 20pt,clip = true, width=.32\linewidth, height=.3\linewidth]
{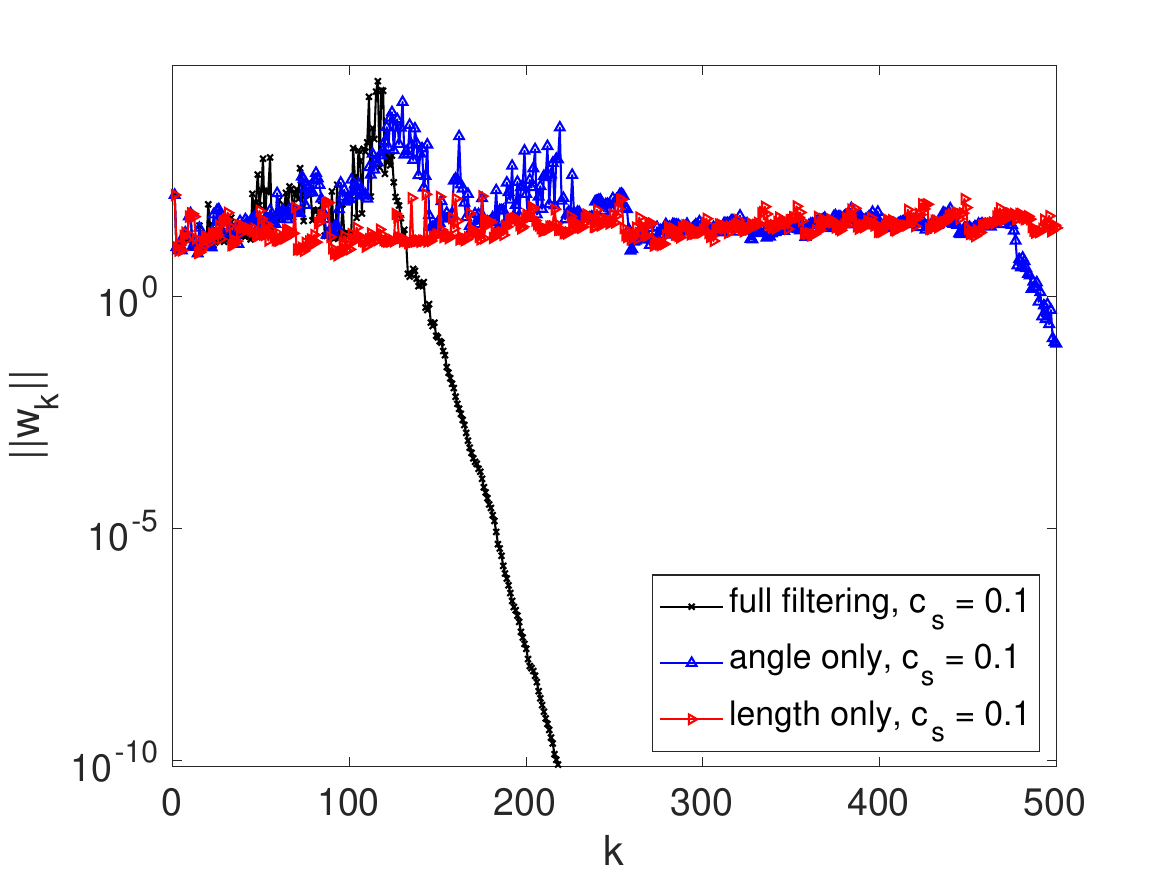}
\caption{Residual convergence for \eqref{eqn:pLap-reg} using update 
\eqref{eqn:pLap1}, comparing full filtering,
angle filtering only and length filtering only.
Left: Lagrange $P_2$ elements with parameter $c_s = 0.1$; center: Lagrange $P_2$ elements with parameter $c_s = 0.5$; right: Lagrange $P_4$ elements with parameter $c_s = 0.1$. }
\label{fig:plap-avsl}
\end{figure}

The next observation we can make is the angle filtering algorithm \ref{alg:afilt}
is important 
although not always sufficient for successful passage through the preasympotic 
regime. Figure \ref{fig:plap-avsl} shows residual convergence using the full 
(length plus angle) filtering approach, just angle filtering and just length filtering.
The plot on the left, with $c_s = 0.1$ for $P_2$ elements
shows the length filtering (with no condition
enforced on the angles between columns of $F_k$) still can work
well in the asympototic regime, and in this case is an adequate if not 
efficient way through the preasymptotic regime.  
The center plot shows $P_2$ elements with 
$c_s = 0.5$, where with a stronger angle condition, the angle filtering
alone is sufficient to achieve asymptotic convergence, whereas the sharper estimate
on the condition number from a larger value of $c_s$ does not force the length filtering
alone to remove enough columns to converge.  The plot on the right for $P_4$ elements
with $c_s =0.1$ reinforces these observations, demonstrating that the combination of
length and angle filtering rather than one strategy alone is more successful, 
in agreement with the theory.

We further see the importance of angle filtering in the preasymptotic regime 
by comparing the first plots in figures \ref{fig:plap-p2-cols}
and \ref{fig:plap-p4-cols} which show the columns of $F_k$ used at each iteration $k$.
For $c_s = 0.1$ for $P_2$ and respectively $P_4$ elements, we see less angle filtering 
in the $P_1$
elements throughout the early stages of the iteration, whereas the $P_4$ elements see more
alignment in the columns of $F_k$ and progress more quickly to the asymptotic regime with
a small angle filtering parameter.  
This is at least in part due to the residual bound \eqref{eqn:tgenm} 
as the angle filtering not only controls the condition number, but also controls the 
scaling and build up of higher-order terms in the residual, which matter more when the 
residual is larger and less as it is smaller.
On the other hand, with a larger angle filtering
parameter $c_s = 2^{-1/2}$, the angle filtering remains heavy throughout the iteration
which ultimately slows convergence in the asymptotic regime.
\begin{figure}
\includegraphics[trim = 56pt 230pt 55pt 240pt,clip = true,scale = 0.50]
{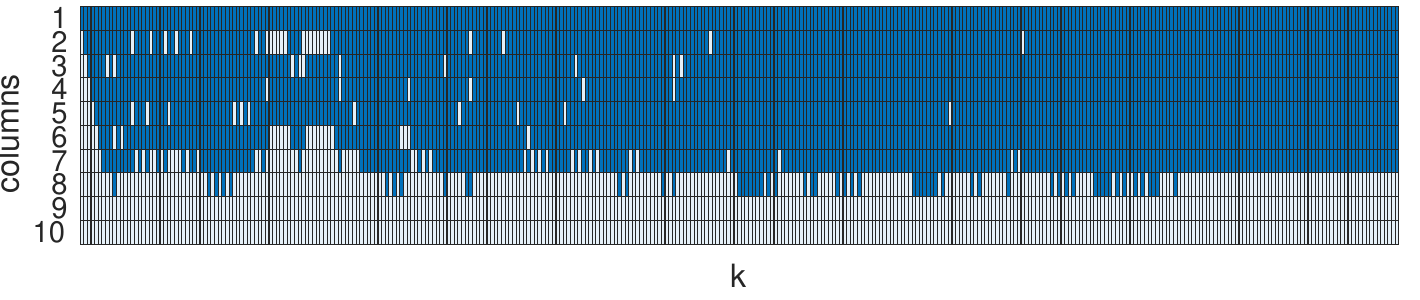}\\
\includegraphics[trim = 217pt 230pt 230pt 240pt,clip = true, scale = 0.50]
{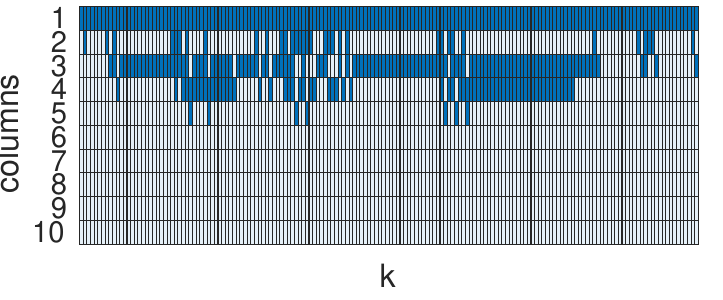}
\includegraphics[trim = 220pt 230pt 220pt 240pt,clip = true, scale= 0.50]
{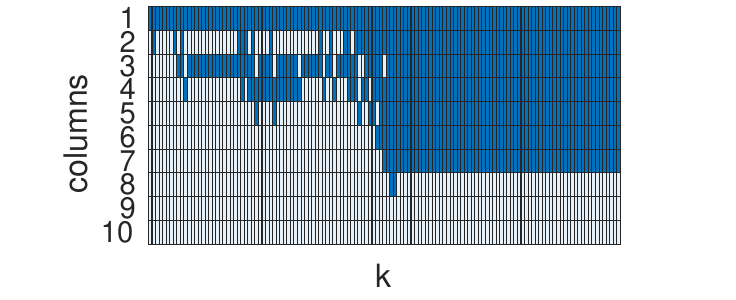}
\caption{The columns used after filtering for problem \eqref{eqn:pLap-reg} using update 
\eqref{eqn:pLap1}, as in table \ref{tabl:pLap-tsvd} with Lagrange $P_2$ elements. 
Column 1 is always used and column 2 may be filtered out in angle filtering.
Top: $c_s = 0.1$,  bottom left: $c_s = 2^{-1/2}$, bottom right: $c_s$ dynamically set by
\eqref{eqn:csdyn}. Each simulation uses algorithmic depth $m=10$.}
\label{fig:plap-p2-cols}
\end{figure}

For this example, where successful solves transition from a poor initial guess through
a preasymptotic regime into an asysmptotic regime, the dynamic choice of angle 
parameter $c_s$ performs the best, as seen in table \ref{tabl:pLap-tsvd}.
This strategy controls the accumulation and scaling
of higher order terms in the residual early in the iteration, and allows greater 
algorithmic depth when the residual terms are small enough not to matter, by which the 
improvement in the first-order term in the residual bound \eqref{eqn:tgenm} 
from addition columns of $F_k$ yields a faster solve.

\begin{figure}
\includegraphics[trim = 170pt 230pt 170pt 240pt,clip = true,scale = 0.60]
{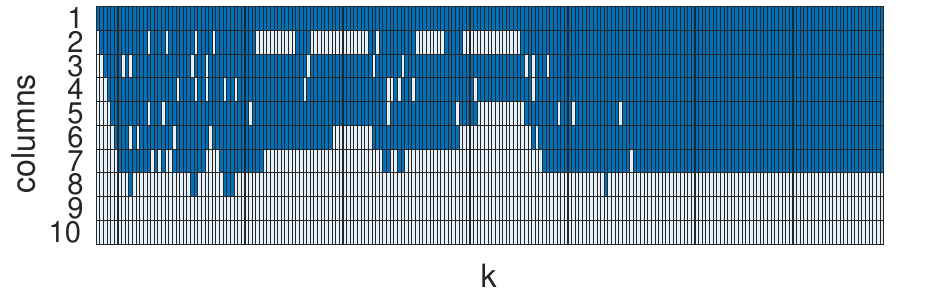}\\
\includegraphics[trim = 150pt 230pt 150pt 240pt,clip = true, scale = 0.60]
{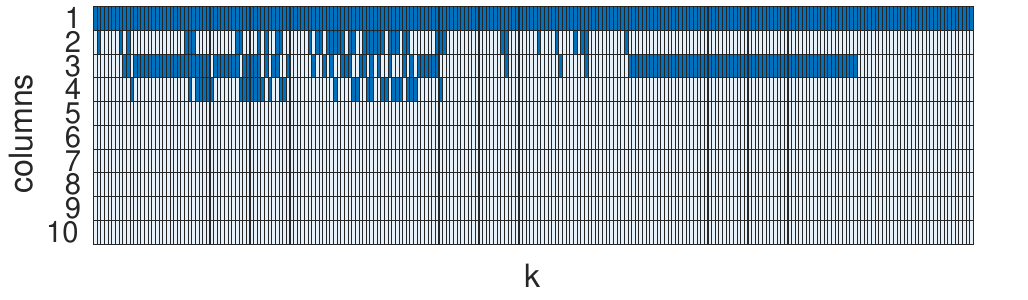}\\
\includegraphics[trim = 208pt 230pt 210pt 240pt,clip = true, scale= 0.60]
{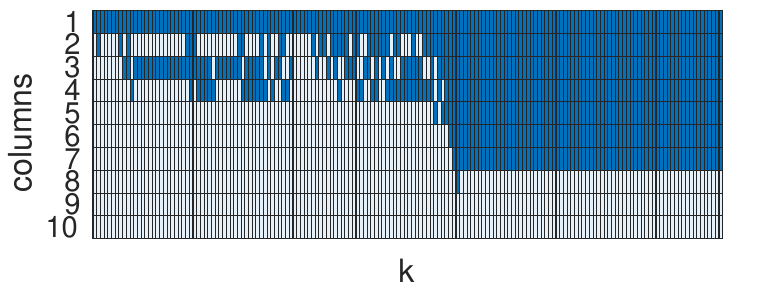}
\caption{The columns used after filtering for problem \eqref{eqn:pLap-reg} using update 
\eqref{eqn:pLap1}, as in table \ref{tabl:pLap-tsvd} with Lagrange $P_4$ elements. 
Column 1 is always used and column 2 may be filtered out in angle filtering.
Top: $c_s = 1$,  middle: $c_s = 2^{-1/2}$, bottom: $c_s$ dynamically set by 
\eqref{eqn:csdyn}. Each simulation uses algorithmic depth $m=10$.}
\label{fig:plap-p4-cols}
\end{figure}

\section{Conclusion}\label{sec:conc}
In this paper we developed a column-filtering strategy to control the condition of 
the least-squares problem in Anderson acceleration.  
We demonstrated theoretically that the method controls the condition number of the 
matrix used in the least-squares problem at each iteration of the algorithm.
The filtering algorithm consists of two phases, one which filters for disparity in 
column length, and the second which imposes a lower bound on the angles between subspaces
spanned by the columns of the least-squares matrix.  The angle filtering approach
has already been demonstrated by the authors to aid convergence in \cite{PR21}, as
it enforces a sufficient condition for controlling the scaling of higher-order terms 
in the residual expansion for nonlinear problems.  The angle filtering is seen to be 
active more in the preasymptotic regime, whereas the presently introduced length 
filtering approach is more active in the asymptotic regime where older columns in the 
least-squares matrix can be many orders of magnitude greater in length than more 
recent columns due to the convergence of the algorithm.  The method is demonstrated 
on a range of problems based on discretized partial differential equations.  

Overall, we see an improvement in convergence properties using the introduced FAA 
algorithm.  A discussion of effective 
parameter ranges and a dynamic strategy for choosing algorithm parameters is included.  Comparing to AA using TSVD, in our numerical tests 
FAA performed always at least as well, but sometimes substantially
better.
In our tests where both FAA and AA with TSVD converged in a similar number of iterations,
FAA maintained a lower condition number for the least-squares problem.
Future work includes integrating the filtering strategies introduced herein with 
dynamic selection of the relaxation parameter to further enhance stability in the
preasymptotic and efficiency in the asymptotic phases of the solution process.

\section{Acknowledgements}
The authors would like to thank the anonymous referees for providing suggestions
that improved the clarity and presentation of this article.
Author SP acknowledges partial support from NSF grant DMS 2011519, and author LR acknowledges partial support from NSF grant DMS 2011490.

\bibliographystyle{siamplain}
\bibliography{graddiv}

\end{document}